\documentclass[]{elsarticle}
\usepackage{amscd,amssymb,amsxtra}
\usepackage{amsfonts}
\usepackage{amsmath}
\usepackage{graphicx}
\usepackage{epsfig}
\usepackage{amssymb}
\usepackage{amstext}

\usepackage{tkz-berge}

\usepackage{calc}
\usepackage{tikz-cd}

\usetikzlibrary{positioning}
\usetikzlibrary{calc}

\usepackage{enumerate}

\newtheorem{theorem}{Theorem}[section]
\newtheorem{lemma}{Lemma}[section]

\newdefinition{remark}{Remark}[section]
\newdefinition{corollary}{Corollary}[section]
\newdefinition{definition}{Definition}[section]
\newdefinition{example}{Example}[section]

\newproof{proof}{Proof}

\def\afrac#1#2{\ifinner {#1}/{#2} \else \frac{#1}{#2} \fi}
\def\1#1{\mbox{\rm{#1}}}

\begin{document}

\def\kob{circle (3pt)} 

\def\ticircw{\tikzset{VertexStyle/.style = {
shape = circle,
fill = black,
inner sep = 0pt,
outer sep = 0pt,
minimum size = 0pt,
draw}}
}

\def\barc{\begin{tabular}{c}} 
\def\earc{\end{tabular}}

\begin{frontmatter}

\title
{Hamiltonian cycles in some family of cubic $3$-connected plane graphs.}
\author{Jan~Florek}
\ead{jan.florek@pwr.edu.pl}

\address{Faculty of Pure and Applied Mathematics,
 Wroclaw University of Science and Technology,
 Wybrze\.{z}e Wyspia\'nskiego 27,
50-370 Wroc{\l}aw, Poland}

\begin{abstract}
Barnette conjectured that all cubic $3$-connected plane graphs with maximum face size at most $6$ are hamiltonian. We provide a method of construction of a hamiltonian cycle (in dual terms) in an arbitrary cubic, $3$-connected plane graph possessing such a face $g$ that every face incident with $g$ has at most $5$ edges and every other face has at most $6$ edges.
\end{abstract}


\end{frontmatter}

\section{Introduction}

For basic definitions and terminology we refer to \cite{flo4} and \cite{flo5} .Let $V(G)$ be the set of all vertices of a plane graph $G$. The set $\mathbb{R}^2 \setminus G$ is open, and its regions are \textit{faces} of $G$. We say that $U \subset V(G)$ \textsl{dominates} all faces of $G$ if every face of $G$ is incident with a vertex belonging to $U$. $G[U]$ denotes the subgraph of $G$ induced by $U$. If $a,b$ are vertices, then $d(a,b)$ denotes the length of the shortest path from $a$ to $b$ in $G$. For a fixed vertex $g \in V(G)$, $V^n(G,g)$ denotes the set of all $v \in V(G)$ such that $d(v, g) = n$.

By a  \textit{plane triangulation} we will mean a simple plane triangulation with at least four vertices.
Hence, it is $3$-connected. Suppose that a pair $(G, g)$ (or $(H, h)$) denotes a plane triangulation $G$ (or $H$) with a fixed vertex $g$ (or~$h$, respectively) belonging to its outer cycle.  We say that $(G, g)$ and $(H ,h)$ are \textit{op-equivalent}
(then we write $(G, g) =(H ,h)$) if there exists an isomorphism $\varphi: G - g \rightarrow H - h$ which preserves the counterclockwise orientation. More precisely: we require that if $x_1 \ldots x_m$ is the outer cycle of the graph $G - g$ with the counterclockwise orientation, then $\varphi(x_1) \ldots \varphi(x_m)$ is the outer cycle of the graph $H - h$ with the counterclockwise orientation. We say that $(G, g)$ and $(H, h)$ are \textit{different} if they are not op-equivalent (then we write $(G, g) \neq (H ,h)$).  The \textsl{height} of $(G, g)$ (denoted as $h(G)$) is the maximal number $n$ such that $V^n(G, g) \neq \emptyset$.

Barnette \cite[Conjecture 5]{flo15}  conjectured that all cubic $3$-connected plane graphs with maximum face size at most $6$ are hamiltonian (see also Gr\"{u}nbaum and Walther \cite{flo13}, p. 380).  Aldred, Bau, Holton and McKay \cite{flo1} used computer search to confirm Barnette's conjecture for graphs up to 176 vertices. Goodey \cite{flo9} (or \cite{flo10}) proved that the hypothesis is satisfied for all such graphs having only faces bounded by $4$ or $6$ edges (by $3$ or $6$ edges, respectively). We give two examples of non-hamiltonian $3$-connected essentially $4$-edge-connected cubic plane graphs (see Fig.~1).  The first one, constructed by Grinberg \cite{flo11} and Tutte \cite{flo12}, has three non adjacent faces bounded by $8$ edges, and other faces bounded by $5$ or $6$ edges. The second one, constructed by Faulkner and Younger \cite{flo6}, has two non adjacent faces bounded by $11$ edges, and other faces bounded by $4$ or $5$ edges. Tutte \cite{flo18} showed, on the other hand, that every $4$-connected plane graph is hamiltonian. Thomassen \cite{flo17} found a shorter proof of this theorem.


\begin{figure}[tbhp!]
\centering

\begin{tabular}{cc}
\begin{tikzpicture}[scale=0.5, line width=0.8pt]

\draw (-4,-4) -- (4,-4) -- (4,4) -- (-4,4) --cycle;
\draw (-3,-3) -- (3,-3) -- (3,3) -- (-3,3) --cycle;
\draw (-4,-4) -- (-3,-3);
\draw (4,-4) -- (3,-3);
\draw (4,4) -- (3,3);
\draw (-4,4) -- (-3,3);
\draw (0,4) -- (0,3);
\draw (-3,0) -- (-2,0);
\draw (3,0) -- (2,0);
\draw (-2,-3) -- (-2,3);
\draw (2,-3) -- (2,3);

\draw (-1,-3) -- (-1,-2);
\draw (-1,-1) -- (-1,1);
\draw (-1,2) -- (-1,3);
\draw (1,-3) -- (1,-2);
\draw (1,-1) -- (1,1);
\draw (1,2) -- (1,3);

\draw (0,-2) -- (0,-1);
\draw (0,2) -- (0,1);

\draw (-1,-2) -- (1,-2);
\draw (-1,-1) -- (1,-1);
\draw (-1,0) -- (1,0);
\draw (-1,1) -- (1,1);
\draw (-1,2) -- (1,2);

\coordinate  (A) at ($(-1,-2)+ (135:0.7071cm)$);

\draw (-1,-2) -- (A) -- (-1,-1);
\draw (-2,-1.5) -- (A);

\coordinate  (B) at ($(1,-2)+ (45:0.7071cm)$);

\draw (1,-2) -- (B) -- (1,-1);
\draw (2,-1.5) -- (B);

\coordinate  (C) at ($(1,2)+ (-45:0.7071cm)$);

\draw (1,2) -- (C) -- (1,1);
\draw (2,1.5) -- (C);

\coordinate  (D) at ($(-1,2)+ (-135:0.7071cm)$);

\draw (-1,2) -- (D) -- (-1,1);
\draw (-2,1.5) -- (D);
\end{tikzpicture}

&

\hspace{1cm}
\begin{tikzpicture}[scale=0.5, line width=0.8pt]

\draw (2,0) -- (2,4) -- (0,5) -- (-2,4) -- (-2,0) -- (-1,-2) -- (1,-2)  --cycle;
\draw (6,0) -- (6,4) -- (0,7) -- (-6,4) -- (-6,0) -- (-2,-4) -- (2,-4) -- cycle;

\draw (1,-2) -- (2,-4);
\draw (-1,-2) -- (-2,-4);
\draw (0,5) -- (0,7);

\draw (2,0) -- (6,0);
\draw (2,1) -- (6,1);
\draw (3,2) -- (5,2);
\draw (2,3) -- (6,3);
\draw (2,4) -- (6,4);

\draw (3,1) -- (3,3);
\draw (5,1) -- (5,3);
\draw (4,0) -- (4,1);
\draw (4,3) -- (4,4);

\draw (-2,0) -- (-6,0);
\draw (-2,1) -- (-6,1);
\draw (-3,2) -- (-5,2);
\draw (-2,3) -- (-6,3);
\draw (-2,4) -- (-6,4);

\draw (-3,1) -- (-3,3);
\draw (-5,1) -- (-5,3);
\draw (-4,0) -- (-4,1);
\draw (-4,3) -- (-4,4);

\end{tikzpicture}
\\
(a) & \hspace{1cm}(b)
\end{tabular}

\label{Fig.1}

\caption{ (a) The Grinberg and Tutte graph, and (b) the Faulkner and Younger graph}

\end{figure}
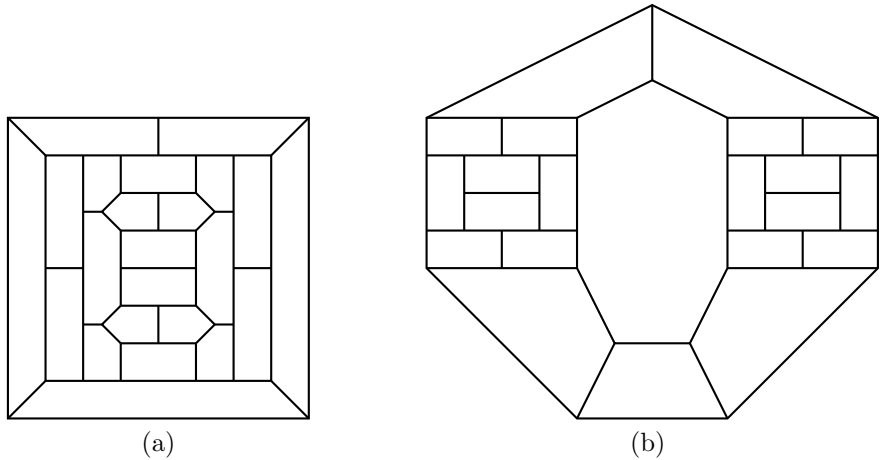


It is known (see Stein \cite{flo16},  Florek \cite{flo7},  Alt, Payne, Schmidt and Wood \cite{flo2}) that the dual version of the Barnette's conjecture is the following: Every plane triangulation with all vertices of degree at most $6$ has an induced tree with the vertex set dominating all faces of the graph. Let $\cal{G}$ be the set of all plane triangulations $(G,g)$ such that every vertex of $G$ different than $g$ has degree at most $6$, and every vertex adjacent to $g$ has degree at most $5$. In this article we prove that every graph in $\cal{G}$ has an induced tree with the vertex set dominating all faces of the graph.

Indeed we will prove more. Let  $(G, g) \in \cal{G}$, and let $c(G) = x_1,\ldots,x_m$ be the outer cycle of  $G - g$ with the counterclockwise orientation (indexed by elements of the cyclic group $\{1, 2, \ldots, m\}$ with the addition modulo $m$).  We say  that $U \subset V(G)$ is \textsl{hamiltonian} in the graph $(G, g)$  if $g \notin U$, $U$ induces a tree and dominates all faces of $(G, g)$.  Hence,  if $x_{i} \notin U$,  then $x_{i-1}$ and $x_{i+1} \in U$. We say that $U$ is \textsl{compatible} with the orientation of $c(G)$ if for every $1 \leqslant i \leqslant m$ the following implication is satisfied: if $x_{i} \notin U$ and  $d_{G}(x_{i}) \leqslant 4$, then $x_{i+2} \in U$. Let $(P, g)$, $(Q, g) \in \cal{G}$ be graphs defined as the graphs $P - g$, $Q - g$ in Fig.~2. We prove the following theorem (see Theorem~\ref{theo3.2}): If $(G, g) \in \cal{G}$ is different than $(P, g)$ and $(Q, g)$, then it has a hamiltonian set of vertices which is compatible with the orientation of $c(G)$. In fact, the proof contains a construction of such hamiltonian set in $(G, g)$  by induction with respect to the height $h(G)$. Notice that $(P, g)$ (or $(Q, g)$) has a hamiltonian set of vertices, marked by circles in Fig.~2, but the set is not compatible with the orientation of $c(P)$ (or $c(Q)$, respectively).

Barnette's other hamiltonicity conjecture \cite{flo3} states that all  cubic $3$-connected bipartite plane graphs are hamiltonian (see also Holton, Manvel and  McKay~\cite{flo14}, or Florek \cite{flo7}, \cite{flo8}).



\section{Generating the family $\cal{G}$}
A simple $2$-connected plane graph with a distinguished face is called a plane \textsl{semi-triangulation} if  every other face, called \textsl{proper face}, is bounded by $3$ edges. If $F$ is a plane semi-triangulation, then $\sigma(F)$ denotes the cycle which bounds the distinguished face. The cycle  $\sigma(F)$ separates the plane into two regions on the plane: the distinguished face, and the other region called the \textsl{proper region} of~$F$. We will make the following assumption: if the proper region of~ $F$ is bounded (unbounded), then $\sigma(F)$ has the counterclockwise (clockwise, respectively) orientation. A plane semi-triangulation is called a \textsl{configuration} of a plane graph~$G$ if it is a subgraph of $G$, and its every proper face is a face of~$G$. Let $\cal{F}$ be a family of all plane semi-triangulations.

A graph is \textsl{generated} by a set of operations  from a starting graph in a class of graphs if it can be constructed by a sequence of these operations from the starting graph and the class is closed under the construction operations. We assume that the starting graph is generated from itself (by the empty set of operations). We give a definition of a replacement operation which replaces a configuration of a plane triangulation by another plane semi-triangulation to obtain a new plane triangulation (Definition \ref{def2.1}--\ref{def2.2}). Next we define an operation~$A$ (Definition \ref{def2.4}) and operations $B$,~$C$ (Definition 2.3) which map graphs of $\cal{G}$, under some conditions, into graphs belonging  to $\cal{G}$. In Theorem~\ref{theo2.1} we prove that every graph in the class $\cal{G}$ is generated by operations of type $A$, $B$ and $C$ from one of graphs: $(G_3, g)$, or $(F_n, g) \in \cal{G}$, $n \geqslant 1$, as the starting graph (Definition \ref{def2.5}--\ref{def2.6}).

\begin{definition}\label{def2.1} If  $\Delta$ is a plane semi-triangulation, $d$ is a fixed vertex of $\sigma(\Delta)$, and $D$ is a fixed  subset of $V(\Delta)$, then a triple $[\Delta, d, D]$ (or $[\Delta, d]$ for $D = \emptyset$) is called a \emph{pattern}.

We say that a patern $[\Delta, d, D]$ is \textsl{$\sigma$-compatible} with a patern $[\Gamma, g, W]$ if there exists an isomorphism $\varphi: \sigma(\Delta)\rightarrow \sigma(\Gamma)$ such that $\varphi(\sigma(\Delta))$ and $\sigma(\Gamma)$ have the same orientation, $\varphi(d) = g$, $\varphi(D\cap V(\sigma(\Delta)) = W \cap V(\sigma(\Gamma))$, and,  if $x y$ is a chord of $\sigma(\Delta)$, then $\varphi(x)\varphi(y)$ is a chord of $\sigma(\Gamma)$.

We write $[\Delta, d, D] = [\Gamma, g, W]$ if there exists an isomorphism $\varphi\colon  \Delta \rightarrow \Gamma$ such that $\varphi(\sigma(\Delta))$ and $\sigma(\Gamma)$ have the same orientation, $\varphi(d) = g$ and $\varphi(D) = W$.

\end{definition}
\begin{definition}\label{def2.2} Let $(G, g)$ be a plane triangulation, 
and let $W$ be a fixed  subset of $V(G)$. Suppose that $[\Gamma_i, g, W \cap V(\Gamma_i)]$, $i \in I$, is a family of patterns such that $\Gamma_i$ is a configuration of $G$, $g \in \sigma(\Gamma_i)$, and the proper regions of $\Gamma_i$ are pairwise disjoint. Assume also that, for every $i \in I$,  $[\Delta_i, d_i, D_i]$ is a pattern which is $\sigma$-compatible with the pattern $[\Gamma_i, g, W \cap V(\Gamma_i)]$. We define a new plane triangulation, denoted $(\gamma(G), g)$, and a set of vertices $\gamma(W)$ in the  following way. For each $i \in I$, we replace $[\Gamma_i, g, W \cap V(\Gamma_i)]$ by $[\Delta_i, d_i, D_i]$; that is, we delete from the graph $G$ all the vertices and edges of $\Gamma_i$ which do not belong to $\sigma(\Gamma_i)$, add the graph $\Delta_i$, and identify cycles $\sigma(\Gamma_i)$ and $\sigma(\Delta_i)$ according to their $\sigma$-compatibility. Then we say that $(\gamma(G),g)$ is obtained from $(G,g)$ by the \textsl{replacement operation} $\gamma$ which we write down symbolically
$$\gamma:  [\Gamma_i, g, W\cap V(\Gamma_i)]   {\xrightarrow[]{i \in I}}  [\Delta_i, d_i, D_i].$$ 
Notice that, for each $i \in I$, we replace $W \cap V(\Gamma_i)$ by $D_i$, and we obtain the following set $\gamma(W) \subset V(\gamma(G))$
$$\gamma(W)= (W \setminus {\bigcup_{i \in I} V(\Gamma_i)}) \cup {\bigcup_{i \in I} D_i}.$$
\end{definition}

We recall that if $(G,g) \in \cal{G}$, then $c(G)$ denotes the outer cycle of the graph $G - g$ with the counterclockwise orientation.
\begin{definition}\label{def2.3} Assume that $(G,g) \in \cal{G}$ and $c(G) = x_1 \ldots x_m$.

(a) Suppose that there is an edge $x_{i}x_{i+1}$ such that $d_G(x_{i})$ and $d_G(x_{i+1})\leqslant~4$. Let  $\Gamma_i$ be a configuration of $G$ such that $x_ix_{i+1}g$ is the only proper face of $\Gamma_i$. Suppose that $[\Delta_1, d]$ is the pattern defined in Fig.~3. Since  $[\Delta_1, d]$ is $\sigma$-compatible with $[\Gamma_i, g]$, we define a replacement operation
$$B_i : [\Gamma_i, g] \rightarrow  [\Delta_1, d]$$
Hence, we obtain a plane triangulation $(B_i(G),g) \in \cal{G}$ from $(G,g)$.

(b) Suppose that there is an edge $x_{i}x_{i+1}$ and a vertex $y_i \in  V^2(G, g)$  such that $x_ix_{i+1}y_i$ is a face in $G$ and $d_G({y_i}) \leqslant 5$. Let  $\Gamma_i$ be a configuration of $G$ such that $x_ix_{i+1}g$ and $x_ix_{i+1}y_i$ are the only proper faces of $\Gamma_i$.  Suppose that $[\Delta_2, d]$ is the pattern defined in Fig.~3. Since $[\Delta_2, d]$ is $\sigma$-compatible with $[\Gamma_i, g]$, we define a replacement operation
$$C_i : [\Gamma_i, g] \rightarrow  [\Delta_2, d].$$
Hence, we obtain a plane triangulation $(C_i(G),g) \in \cal{G}$ from $(G,g)$.
 
(c) Suppose that there is an edge $x_{i}x_{i+1}$, and vertices $y_{i}$, $y_{i+1} \in V^{2}(G, g)$ such that $x_{i}x_{i+1}y_i$, $y_{i}y_{i+1}x_{i+1}$ are faces in $G$, and $d_G({y_i})$, ${d_G({y_{i+1}}) \leqslant 5}$. Let  $\Gamma_i$ be a configuration of $G$ such that $x_{i}x_{i+1}g$, $x_{i}x_{i+1}y_i$, $y_{i}y_{i+1}x_{i+1}$ are the only proper faces of $\Gamma_i$. Suppose that $[\Delta_3, d]$ is the pattern defined in Fig.~3. Since $[\Delta_3, d]$ is $\sigma$-compatible with $[\Gamma_i, g]$, we define a replacement operation
$$\bar{B}_i : [\Gamma_i, g] \rightarrow  [\Delta_3, d].$$
Hence, we obtain a plane triangulation $(\bar{B}_i(G),g) \in \cal{G}$ from $(G,g)$ (see Fig.~7).

(d) Suppose that there is an edge $x_{i}x_{i+1}$, vertices $y_i$, $y_{i+1} \in V^{2}(G, g)$, and a vertex $v_i \in V^3(G, g)$ such that $x_ix_{i+1}y_i$, $y_iy_{i+1}x_{i+1}$, $y_iy_{i+1}v_i$ are faces in $G$ and $d_G({v_i}) \leqslant 5$. Let~$\Gamma_i$ be a configuration of $G$ such that $x_ix_{i+1}g$, $x_ix_{i+1}y_i$, $y_iy_{i+1}x_{i+1}$, $y_iy_{i+1}v_i$ are the only proper faces of $\Gamma_i$. Suppose that $[\Delta_4, d]$ is the pattern defined in Fig.~3. Since $[\Delta_4, d]$ is $\sigma$-compatible with $[\Gamma_i, g]$, we define a replacement operation
$$\bar{C}_i : [\Gamma_i, g] \rightarrow  [\Delta_4, d].$$
Hence, we obtain a plane triangulation $(\bar{C}_i(G),g) \in \cal{G}$ from $(G,g)$ (see Fig.~8).

The operation $B_i$, $C_i$, $\bar{B}_i$ and $\bar{C}_i$ depends on the choice of the edge $x_ix_{i+1}$.  Any such operation will be called of \textsl{type} $B$, $C$, $\bar{B}$ and $\bar{C}$, respectively, and the index $i$ will be omitted.
\end{definition}



\begin{definition}\label{def2.4} Let $(G,g) \in \cal{G}$ and suppose that $c(G) = y_1 \ldots y_m$.

(a) We first delete all edges $gy_i$ from $G$. Next, we add a cycle $x_1 \ldots x_m$, with the counterclockwise orientation, in such a way that $G - g$  is contained in the  bounded, and $g$ in the unbounded region of the cycle. At last, we add edges $x_{i}y_i$, $x_{i+1}y_{i}$, $gx_i$, for $1\leqslant i \leqslant m$. We obtain a plane triangulation, denoted  as $A(G)$, such that $(A(G),g) \in \cal{G}$.
 Hence, we define an operation $A : {\cal{G}} \rightarrow {\cal{G}} : (G,g) \rightarrow  (A(G),g)$. We denote $A(A(G))$ briefly by $A^2(G)$.

(b)
For $h(G) > 1$, $G^{-1}$ is a plane graph obtained from $G$ by deleting all vertices of $c(G)$ and adding all edges $gv$, for $v~\in V^2(G, g)$ . Let us observe that $G^{-1}$ may be not simple. If $(G^{-1},g) \in \cal{G}$ we  define $G^{-2} = (G^{-1})^{-1}$.
\end{definition}
\begin{definition} \label{def2.5}
We define a graph $(G_n, g) \in \cal{G}$, for $n \geqslant 3$, (or $(J, g) \in \cal{G}$) as the graph $G_n - g \in \cal{F}$ (or $J - g \in \cal{F}$) shown in Fig.~4.
\end{definition}


\begin{definition} \label{def2.6} We define a graph $(F_n,g) \in \cal{G}$, for $n\geqslant 1$, as the graph ${F_n -g \in \cal{F}}$ shown in Fig.~5.

Let ${\cal{G}}_n$, for $n \geqslant 1$, (or ${\cal{G}}_0$) be the family of all graphs in ${\cal{G}}$ which are generated by operations of type $A$, $B$, $C$ from the graph $(F_n,g)$ (or $(G_3, g)$, respectively).
\end{definition}
\begin{example}\label{exam2.1}
Notice that $(P, g)$ (or $(Q, g)$) is generated by operations of type $C$ from $(A(G_3), g)$ (or $(A(F_1), g)$, respectively).
\end{example}

\begin{lemma}\label{lem2.1}
 If $(G,g) \in \cal{G}$ and all vertices of the cycle $c(G)$ are of degree $5$, then one of the following conditions is satisfied:
\begin{enumerate}
\item [(1)]
$(G^{-1},g) \in \cal{G}$ and  $(G,g) = (A(G^{-1}),g)$,
\item [(2)]
$(G, g) = (F_n, g)$ for some $n \geqslant 3$.
\end{enumerate}
\end{lemma}

\begin{proof}
Suppose that $c(G) = {x_1}\ldots{x_m}$. Let us denote by $y_i \neq g$, $1 \leqslant i \leqslant m$, a vertex of the graph $G$ such that
\begin{enumerate}
\item [($i_1$)]
 $x_{i}x_{i+1}y_{i}$ is a face in $G$.
\end{enumerate}
Since  $x_{i+1}$ has degree $5$, for $1 \leqslant i \leqslant m$, we have:
\begin{enumerate}
\item [($i_2$)]
 $y_iy_{i+1}x_{i+1}$ is a face in $G$.
\end{enumerate}
We will first prove the following condition:
\begin{enumerate}
\item [($i_3$)]
 $\{x_i: 1 \leqslant i \leqslant m \} \cap \{y_i: 1 \leqslant i \leqslant m \}= \emptyset$.
\end{enumerate}
Suppose, on the contrary, that $y_j = x_k$ for some $j, k$. Hence, by~$(i_1)$, $x_{j}x_{j+1}x_{k}$ is a face in $G$. Then,  $x_{k+1} \neq x_j$ and $x_{k-1} \neq x_{j+1}$, because $x_j$ and $x_{j+1}$ doesn't have degree $3$ in $G$. Hence, $x_{k-1}x_kx_{j+1}$ and $x_kx_{k+1}x_j$ are faces in $G$, because $x_k$ has degree $5$.  The same is true for $x_{k-1}x_kx_{j+1}$ in place of $x_{j}x_{j+1}x_{k}$. Thus,  $x_{j+1}x_{j+2}x_{k-1}$ is a face in $G$. We can continue in this fashion, to obtain an infinite sequence 
of  faces in $G$, which is a contradiction. Hence $(i_3)$ holds.

If $y_i \neq y_j$, for every $i \neq j$, then if follows from $(i_2)$ that ${y_1}\ldots{y_m}$ is a cycle. Hence $(1)$ holds.

If $y_p = y_s$, for some $p < s$, then, by $(i_2)$, $1 < s - p < m - 1$. We prove the following implications  (see Fig.~6):
\begin{enumerate}
\item [($i_4$)]
 if $y_p = y_s$, $p < s$, then $y_{p+1} = y_{s-1}$, $y_{p-1} = y_{s+1}$, and $y_p$ has degree $6$.
\item [($i_5$)]
if $y_p = y_s$ and $s - p = 2$ (or $s - p = m - 2$), then $y_{p+1}$ (or $y_{p-1}$, respectively) has degree $3$.
\end{enumerate}
Indeed, if $y_p = y_s$, then, by $(i_1)$--$(i_3)$, the vertex $y_p = y_s$ is incident with at least six faces: $y_{p-1}y_{p}x_{p}$, $x_{p}y_{p}x_{p+1}$, $x_{p+1}y_{p}y_{p+1}$ and $y_{s+1}y_{s}x_{s+1}$, $x_{s+1}y_{s}x_{s}$, $x_{s}y_{s}y_{s-1}$.  Thus, $y_{p-1}y_{p} = y_{s+1}y_{s}$ and $y_{p}y_{p+1} = y_{s}y_{s-1}$, because $y_p = y_s$ has degree at most~$6$.  Hence, $(i_4)$ holds.

If $y_p = y_{p+2}$, then, by $(i_1)$--$(i_3)$, the vertex $y_{p+1}$ is incident with exactly three faces: $y_{p}y_{p+1}x_{p+1}$, $x_{p+1}y_{p+1}x_{p+2}$, and $x_{p+2}y_{p+1}y_{p+2} =  x_{p+2}y_{p+1}y_{p}$. Thus, $y_{p+1}$  has degree $3$. Similarly, if $y_1 = y_{m-1}$ (or $y_2 = y_m$), then  $y_m$  (or $y_1$) has degree~$3$. Hence, $(i_5)$ holds.

Finally, $(2)$ follows from conditions $(i_2)$, $(i_4)$  and $(i_5)$.
\qed \end{proof}



\begin{theorem}\label{theo2.1}
 If $(G, g) \in \cal{G}$, then it satisfies one of the following conditions:
\begin{enumerate}[\upshape(1)]
\item
$(G^{-1}, g) \in \cal{G}$ and $(G, g)$ is generated by operations of type $C$ and $B$ from $(A(G^{-1}), g)$,
\item
$(G, g)$ is generated by operations of type $C$ and $B$ from $(F_n, g)$, for some $n \geqslant 1$,
\item
$h(G)= 1$ and $(G, g)$ is generated by operations of type $B$ from $(G_3, g)$.
\end{enumerate}
 \end{theorem}

 \begin{proof}
Let $(G, g) \in \cal{G}$ and $c(G) = {x_1}\ldots{x_m}$. We proceed by induction with respect to the order of the graph.  By Lemma \ref{lem2.1}, it is sufficient to consider the following two cases:
\begin{enumerate}
\item [(i)] $d_G(x_2) = 3$,
\item [(ii)] $d_G(x_2) = 4$ and $d_G(x_i) \geqslant 4$, for every $i$.
 \end{enumerate}

Case (i) If $d_G(x_1) \geqslant 4$ and $d_G(x_3) \geqslant 4$, then $(G - x_2, g) \in \cal{G}$. Hence, by induction, $(G - x_2, g)$ satisfies one of the conditions (1)--(3). Since $(G^{-1},g) = ((G - x_2)^{-1}, g)$ and $(G, g) = (B(G - x_2), g)$, $(G, g)$ satisfies one of the conditions (1)--(3).
If $d_G(x_1) = 3$ or $d_G(x_3) = 3$, then $(G, g) = (G_3, g)$.

Case (ii) The vertices $x_1$, $x_2$ and $x_3$ are adjacent to a vertex $y$ different than  $g$. If $y$ is a vertex of the cycle $c(G)$, then $d_G(y) \leqslant 5$. It follows that $d_G(x_1) = 3$ or $d_G(x_3) = 3$. We obtain a contradiction, hence $y \in V^2(G, g)$.

If vertices $x_1$, $x_3$ are not adjacent, then $H = (G - x_2) + x_1x_3 \in \cal{G}$ and the vertex $y \in V^2(H, g)$. Thus, by induction, $(H, g)$ satisfies condition (1) or (2). Since $(G^{-1}, g) = (H^{-1}, g)$ and $(G, g) = (C(H), g)$, $(G, g)$ satisfies condition (1) or~(2).

If vertices $x_1$, $x_3$ are adjacent, then $(G, g) = (F_1, g)$ or $(G, g) = (F_2, g)$.
\qed \end{proof}
\begin{corollary}\label{coro2.1}
Let  $(G, g) \in {\cal {G}}_0$.
\begin{enumerate}
\item[(1)] If $h(G) = 1$, then it is op-equivalent to one of the following graphs:
$$(J, g), \ (G_n, g), \hbox{ or } (r(G_{2n}), g), \hbox{ for } n \geqslant 3,$$
where $r(G)$ is a mirror reflection of $G$ on the plane such that $r(g) = g$.
\item[(2)] If $h(G) = 2$, then it is generated by operations of type $C$ and $B$ from one of graphs:
  $$(A(J), g), \ (A(G_n), g), \hbox{ or } (r(A(G_{2n})), g), \hbox{ for } n \geqslant 3.$$
\end{enumerate}
\begin{proof}
The proof follows from Theorem \ref{theo2.1}, because every graph $(G, g) \in \cal{G}$ which is generated by operations of typy $B$ from $(G_3, g)$ is op-equivalent to one of the following graphs:  
$$(J, g), \ (G_n, g), \hbox{ or } (r(G_{n}), g), \hbox{ for } n \geqslant 3.$$
Notice that 
$$(r(G_4), g)  = (G_4, g), \, (r(G_{2n+1}), g) = (G_{2n+1}, g), \hbox{ for } n \geqslant 1,$$
and
$$(r(G_{2n}), g) \neq (G_{2n}, g), \hbox{ for } n \geqslant 3.$$
\qed \end{proof}
\end{corollary}
\begin{corollary}\label{coro2.2}
For every $(G,g) \in \cal{G}$, the following equivalences hold:
\begin{enumerate}
\item[(1)]
$(G,g) \in {\cal{G}}_{0}$ $\Leftrightarrow$  $V^{h(G)}(G, g)$ is the set of vertices of a cycle in $G$,
 \item[(2)]
$(G,g) \in {\cal{G}}_{n}$, $n \geqslant 1$ $\Leftrightarrow$ 
$V^{h(G)}(G, g)$ induces a path of length $n-1$ in $G$.
\end{enumerate}
\end{corollary}
\begin{proof}
The proof follows from Theorem \ref{theo2.1} by induction with respect to the height $h(G)$.
\qed \end{proof}
\begin{corollary}\label{coro2.3}
$\{ {\cal{G}}_0, {\cal{G}}_1, {\cal{G}}_2, \ldots \}$ is a partition of $\cal{G}$.
\end{corollary}
\begin{proof}
The proof follows from Corollary \ref{coro2.2}. 
\qed \end{proof}
\begin{remark}\label{rem2.1}
 Let $(G,g) \in \cal{G}$. If $k < h(G)$  ($k = h(G)$), then, by Theorem~\ref{theo2.1} and  Corollary \ref{coro2.2}, $V^k(G, g)$ is the set of vertices of a cycle (a cycle or a path, respectively) in $G$, which will be denoted by $c^k(G)$. We assume that the cycle $c^k(G)$ has the counterclockwise orientation, and it is indexed by elements of the cyclic group $\{1, 2, \ldots, |c^k(G)|\}$ with the addition modulo $m$.
\end{remark}

\begin{theorem}\label{theo2.2}
 If $(G,g)\in {\cal{G}}_n$, $n \geqslant 0$, is a graph of height $h(G) \geqslant 3$, then one of the following conditions is satisfied:.
\begin{enumerate}[\upshape(1)]
\item
$(G^{-2}, g) \in {\cal{G}}_n$ and  $(G, g)$
is generated by operations of type $\bar{C}$, $\bar{B}$, $C$, $B$ from $(A^2(G^{-2}), g)$,
\item
$n \geqslant 1$ and $(G, g)$ is generated by operations of type $\bar{C}$, $\bar{B}$, $C$, $B$ from $(A(F_n), g)$.
\end{enumerate}
\end{theorem}

\begin{proof}  Since  $(G, g) \in {\cal{G}}_n$ and $h(G) \geqslant 3$, by Theorem \ref{theo2.1} and Corollary \ref{coro2.3}, $(G^{-1}, g) \in {\cal{G}}_n$ and $(G, g)$ is generated by operations of type $C$ and $B$ from $(A(G^{-1}), g)$. Furthermore, since $h(G^{-1}) \geqslant 2$, one of the following cases holds:
\begin{enumerate}
 \item[(a)]
$(G^{-2}, g) \in {\cal{G}}_n$ and $(G^{-1}, g)$ belongs to the family $\cal {S}$ of all graphs which are generated by operations of type $C$ and $B$ from $(A(G^{-2}), g)$,
 \item[(b)]
$n \geqslant 1$ and $(G^{-1}, g)$ belongs to the family ${\cal {S}}_0$ of all graphs which are generated by operations of type $C$ and $B$ from $(F_n, g)$.
\end{enumerate}

Case (a) Let $\cal {T} \subset \cal{G}$ be the family of all graphs which are generated by operations of type $\bar{C}$, $\bar{B}$ from $(A^2(G^{-2}),g)$.  It is sufficient to prove the following:
$$\hbox{if }\ (K, g) \in {\cal {S}}, \hbox{ then } (A(K), g) \in {\cal {T}}.$$
We proceed by induction with respect to the order of a graph. Let $(K, g) \in \cal {S}$. We can assume that $(K, g) \neq (A(G^{-2}), g)$. Hence, there is a graph $(H,g) \in \cal {S}$   and a replacement operation  
$$B_i : [\Gamma_i, g] \rightarrow  [\Delta_1, d], \hbox{ or } C_i : [\Gamma_i, g] \rightarrow  [\Delta_2, d],$$ 
such that 
\begin{enumerate}
 \item[(i)]
$(K, g) = (B_{i}(H), g)$, or $(K, g) = (C_{i}(H), g)$,
\end{enumerate}
where  $c(H) = {y_1}\ldots{y_m}$, and $\Gamma_i$ is a configuration of $H$ bounded by the cycle $y_{i}gy_{i+1}$ in Fig.~7, or by the cycle $y_{i}gy_{i+1}v_{i}$ in Fig.~8, respectively. 

It is easy to see from the definitions of the operations $A$, $B_{i}$, $C_{i}$, $\bar{B}_{i}$, $\bar{C}_{i}$ that 
$$(A(B_{i}(H)), g) = (\bar{B}_{i}(A(H)), g)\ \hbox{ (see Fig.~7)},$$
 or
  $$(A(C_{i}(H)), g) = (\bar{C}_{i}(A(H)), g)\ \hbox{ (see Fig.~8)}.$$
 Hence, by (i), we obtain
\begin{enumerate}
 \item[(ii)]
$(A(K), g) = (\bar{B}_{i}(A(H)), g)$,
or $(A(K), g) = (\bar{C}_{i}(A(H)), g)$.
\end{enumerate}

Since $H$ has fewer vertices than $K$, we can assume, that $(A(H), g) \in \cal {T}$. Hence, by (ii), $(A(K), g) \in \cal {T}$.




Case (b) Let ${\cal {T}}_0 \subset \cal{G}$ be the family of all graphs which are generated by operations of type $\bar{C}$, $\bar{B}$ from  $(A(F_n), g)$. It is sufficient to prove the following:
$$\hbox{if }\ (K, g) \in {\cal {S}}_0, \hbox{ then } (A(K), g) \in {\cal {T}}_0.$$
The proof is analogous to that in (a).
\qed \end{proof}

\section{Hamiltonian sets in graphs from the family $\cal {G}$}
Let $(G,g) \in \cal {G}$. We recall that $c(G) = x_1,\ldots,x_m$ is the outer cycle of  ${G - g}$ with the counterclockwise orientation. We say that $U$ is \textsl{($-$)compatible} with the orientation of $c(G)$, if for every $1 \leqslant i \leqslant m$ the following implication is satisfied: if $x_{i} \notin U$ and  $d_{G}(x_{i}) \leqslant 4$, then $x_{i-2} \in U$. We say that $U \subset V(G)$ is \textsl{($\pm$)compatible} with the orientation of $c(G)$, if $U$ is compatible as well as ($-$)compatible with the orientation of $c(G)$. 
Notice that the graph $(A_4, g)$ defined in Fig.~15 (or the graph $(H_7, g)$ in Fig.~20) doesn't have any hamiltonian set of vertices which is ($\pm$)compatible with the orientation of $c(A_4)$ (or $c(H_7)$, respectively).

Two vertices $u$ and $v$ of the graph $G  - g$ are called  \textsl{similar}  if there exists automorphism of $G  - g$ (rotation) which maps $u$ to $v$, and preserves the counterclockwise orientation of $c(G)$.

We denote the numbers of vertices and edges in a plane graph $H$ by $|H|$ and $e(H)$, respectively. It is well known that if $H$ is connected, then it is a tree if and only if $e(H) = |H| - 1$. 

\begin{definition}\label{def3.1}
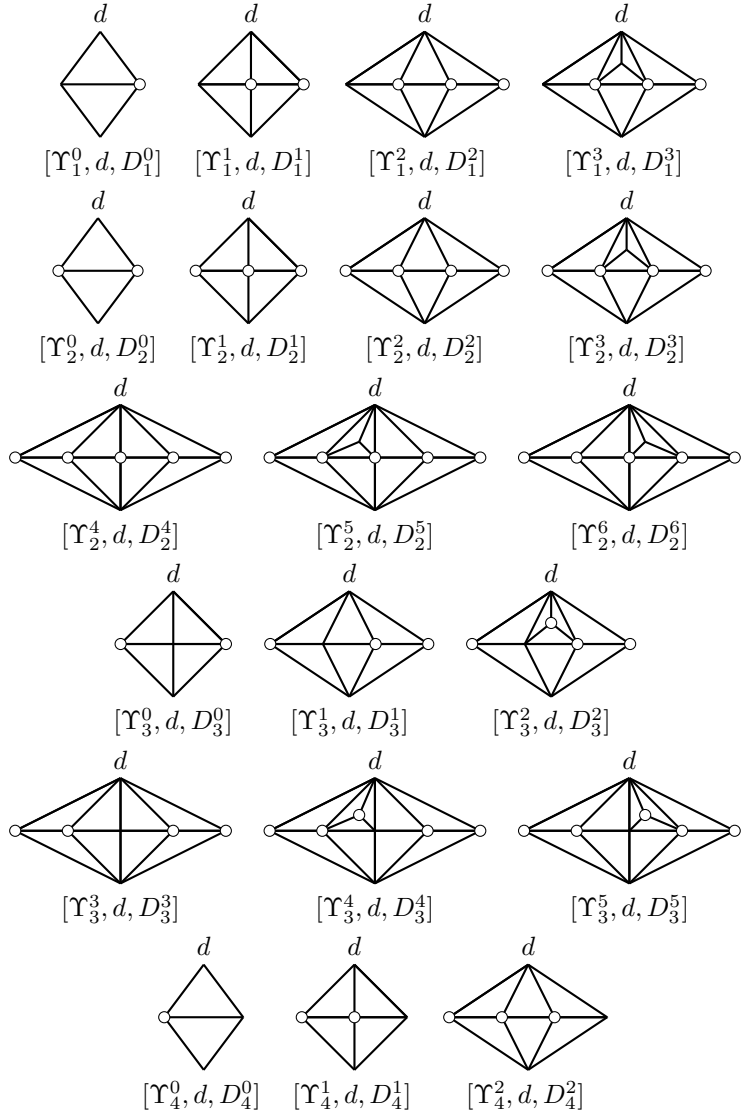
For $1 \leqslant k \leqslant 4$ and $0 \leqslant j \leqslant n(k)$, where $n(1) = 3$, ${n(2) = 6}$, $n(3) = 5$, and $n(4) = 2$ we define a pattern
$[\Upsilon^j_k, d, D^j_k]$, where $\Upsilon^j_k$ is a plane semi-triangulation in Fig.~9, $d$ is a fixed vertex in $\sigma(\Upsilon^j_k)$, and $D^j_k$ is a fixed set of vertices marked by the circles in the graph $\Upsilon^j_k$.


\end{definition}
The following lemma follows directly from Definition \ref{def3.1}.
\begin{lemma}\label{lem3.1}
For $1 \leqslant k \leqslant 4$ and $0 \leqslant j \leqslant n(k)$, where  $n(1) = 3$, $n(2) = 6$, $n(3) = 5$, $n(4) = 2$, the following conditions are satisfied:
\begin{enumerate}[\upshape(1)]
\item
$[\Upsilon^j_k, d, D^j_k]$ is $\sigma$-compatible with $[\Upsilon^0_k, d, D^0_k]$,
\item
$D^j_k$ dominates all faces of $\Upsilon^j_k$, and $d \notin D^j_k$.
\end{enumerate}
\end{lemma}
We will define a replacement operation $\omega$ which will be used in Lemma \ref{lem3.2}, Theorem~\ref{theo3.1} and Lemma \ref{lem3.8}.

\begin{definition}\label{def3.2}
Suppose that $(H, g) \in \cal{G}$ is a graph of the height $h(H) \geqslant 2$, $c(H) = z_1 \ldots z_m$ and $c^2(H) = t_1 \ldots t_n$. Let $X$ be a hamiltonian set in $(H, g)$ which is $(-)$compatible with the orientation of $c(H)$. Let us denote
 \begin{align*}
M &= \{ 1\leqslant i \leqslant n: t_i \notin X \hbox{ and } d_H(t_i) \leqslant 5 \},
\\
M_1 &= \{ i \in M:  d_H(t_i) \geqslant 4, \hbox{ there is exactly one face } z_{j}z_{j+1}t_i,
\hbox{ for some }
\\
&\quad j = j(i), \hbox{such that } z_j \in X, z_{j+1} \notin X, \hbox{ and } d_H(z_j) = d_H(z_{j+1}) = 5  \},
\\
M_2  &= \{ i \in M:  \hbox{ there is exactly one face }z_{j}z_{j+1}t_i, \hbox{ for some } j = j(i), \hbox{ such that}
\\
&\quad z_j, z_{j+1} \in X ,\hbox{ and } d_H(z_j), d_H(z_{j+1}) = 4 \hbox{ or } 5 \},
\\
M_3 &= \{ i \in M:  d_H(t_i) \geqslant 4, \hbox{ there is exactly one pair of faces } z_{j-1}z_{j}t_i, z_{j}z_{j+1}t_i,
\\
&\quad\hbox{for some } j = j(i), \hbox{ such that } z_{j-1}, z_{j+1} \in X, z_j \notin X, d_H(z_j)= 4, \hbox{and }
\\
&\quad d_H(z_{j-1}) = d_H(z_{j+1}) = 5 \}.
\\
M_4 &= \{ i \in M\:  d_H(t_i) \geqslant 4, \hbox{there is exactly one face } z_{j}z_{j+1}t_i, \hbox{ for some }
\\
&\quad j = j(i), \hbox{such that } z_j \notin X,  z_{j+1} \in X ,\hbox{ and } d_H(z_j) = d_H(z_{j+1}) = 5 \}.
\end{align*}
We assume that
\begin{enumerate}
 \item[(1)]
$M_1$, $M_2$, $M_3$, $M_4$ are pairwise disjoint.
 \end{enumerate}

Let $\Omega_i$, $i \in {M_1 \cup M_2 \cup M_4}$, be a configuration of  $H$ such that $z_{j(i)}z_{j(i)+1}g$ and $z_{j(i)}z_{j(i)+1}t_i$ are the only proper faces of $\Omega_i$. By Definition \ref{def3.1} of the pattern $[\Upsilon^0_k, d, D^0_k]$, for $k = 1, 2, 4$, we obtain:
\begin{enumerate}
 \item[(2)]
 $\left\{
 \begin{array}{l}
[\Omega_i, g, X \cap V(\Omega_i)] = [\Upsilon^0_1, d, D^0_1], \hbox{ for } i \in M_1,
\\[4pt]
[\Omega_i, g, X \cap V(\Omega_i)] = [\Upsilon^0_2, d, D^0_2], \hbox{ for } i \in M_2,
\\[4pt]
[\Omega_i, g, X \cap V(\Omega_i)] = [\Upsilon^0_4, d, D^0_4], \hbox{ for } i \in M_4.
 \end{array}
 \right.$
 \end{enumerate}

Let $\Omega_i$, $i \in  M_3$, be a minimal configuration of  $H$ such that $z_{j(i)}$ is the only vertex in the proper region of $\Omega_i$. By Definition \ref{def3.1} of the pattern $[\Upsilon^0_3, d, D^0_3]$ we obtain:
\begin{enumerate}
 \item[(3)]
 $[\Omega_i, g, X \cap V(\Omega_i)] = [\Upsilon^0_3, d, D^0_3]$, for $i \in M_3$.
 \end{enumerate}

Let $\Omega(H,X)$ be the set of all functions $\bar{\omega}: M_1 \cup \ldots \cup M_4 \to \{ 0,1,\ldots, 6 \}$ satisfying the following conditions:

$\bar{\omega} (i) \leqslant 6$, for $d_H(t_i) = 3$ and $i \in M_2$,

$\bar{\omega} (i) \leqslant 5$, for $d_H(t_i) = 4$ and $i \in M_3$,

$\bar{\omega} (i) \leqslant 3$, for $d_H(t_i) = 4$ and $i \in M_1 \cup M_2$,

$\bar{\omega} (i) \leqslant 2$, for $d_H(t_i) = 5$ and $i \in M_3$,

$\bar{\omega} (i) \leqslant 2$, for $d_H(t_i) = 4$ and $i \in M_4$,

$\bar{\omega} (i) \leqslant 1$, for $d_H(t_i) = 5$ and $i \in M_1 \cup M_2 \cup M_4$.
\\
Moreover, let $\Omega^*(H,X)$ be the set of all functions $\bar{\omega} \in  \Omega(H,X)$ such that

$\bar{\omega} (i) \leqslant 2$, for $d_H(t_i) = 4$ and $i \in M_1$.
\\

Fix $\bar{\omega} \in \Omega(H,X)$. It follows from (2) and (3) that proper regions of $\Omega_i$, for $i \in {M_1 \cup \ldots \cup M_4}$, are pairwise disjoint. Hence, by Lemma~\ref{lem3.1}(1), we can define a replacement operation
 $$\omega: [\Omega_i, g,  X \cap V(\Omega_i)]
 \ {\xrightarrow[]{i \in M_k, \ k = 1, \ldots, 4}} \
 [ \Upsilon^{\bar{\omega}(i)}_{k}, d, D^{\bar{\omega}(i)}_{k}],$$
to obtain a plane triangulation $(\omega(H),g)$ from  $(H,g)$, and the set
$$\omega(X)= (X \setminus \bigcup_{i \in M_1 \cup \ldots \cup M_4} V(\Omega_i)) \cup  \bigcup_{k = 1, \ldots, 4} \\ \bigcup_{i \in M_k} D^{\bar{\omega}(i)}_k \subset V(\omega(H)).$$

If $M_1 \cup \ldots \cup M_4 = \emptyset$ we put $(\omega(H),g) =(H,g)$ and $\omega(X) = X$.
\end{definition}

\begin{lemma}\label{lem3.2}
Let $(H, g) \in \cal{G}$ be a graph of height $h(H) \geqslant 2$, and let $X$ be a hamiltonian set in $(H, g)$ which is $(-)$compatible with the orientation of $c(H)$. Let us define sets $M_1, \ldots, M_4$ satisfying condition (1) of Definition \ref{def3.2}, the set of functions $\Omega(H,X)$, and the replacement operation $\omega$ for $\bar{\omega} \in \Omega(H,X)$ as in Definition \ref{def3.2}.

For every $\bar{\omega} \in \Omega(H, X)$, $(\omega(H), g) \in \cal{G}$, and $\omega(X)$ is a hamiltonian set in $(\omega(H), g)$ which is $(-)$compatible with the orientation of $c(\omega(H))$.

Moreover, if $X$ is  $(\pm)$compatible with the orientation of $c(H)$, then  for every $\bar{\omega} \in \Omega^*(H, X)$, $\omega(X)$ is $(\pm)$compatible with the orientation of $c(\omega(H))$.
\end{lemma}

\begin{proof}
Let $\bar{\omega} \in \Omega(H,X)$. Since $M_1$, $M_2$, $M_3$, $M_4$ are pairwise disjoint, we have ${d_{\omega(H)}(t) \leqslant 6}$, for every $t \in c^2(\omega(H)) = c^2(H)$. Hence, $(\omega(H), g) \in {\cal{G}}$.

Since $X$ dominates all faces in $H$, by   Lemma \ref{lem3.1}(2), $\omega(X)$ dominates all faces in $\omega(H)$. Certainly, $\omega(X)$ induces a tree in $\omega(H)$, because  $X$ induces a tree in $H$. Since  $X$ is $(-)$compatible with the orientation of $c(H)$, $\omega(X)$ is $(-)$compatible with the orientation of $c(\omega(H))$, by the definition of patterns $[\Upsilon^j_k, d, D^j_k]$. Hence, Lemma~\ref{lem3.2} holds.
\qed \end{proof}

\begin{definition}\label{def3.3}
For $k = 1, 2, 3$, $0 \leqslant j \leqslant  n(k)$, where $n(1) = 3$, $n(2) = n(3) = 6$, we define a pattern
$[\Delta^j_k, d, D^j_k]$, where $\Delta^j_k$ is a plane semi-triangulation in Fig.~10 (for $k = 1, 2$) or Fig.~11 (for $k = 3$), $d$ is a fixed vertex in $\sigma(\Delta^j_k)$, and $D^j_k$ is a fixed set of vertices marked by the circles in the graph $\Delta^j_k$.



\end{definition}
The following lemmas are simple consequences of Definition \ref{def3.3}.

\begin{lemma}\label{lem3.3}
For $k = 1,  2,  3$, $0 \leqslant  j \leqslant n(k)$, where $n(1) = 3$, $n(2) = n(3) =6$, the following conditions are satisfied:
\begin{enumerate}[\upshape(1)]
\item
$[\Delta^j_k, d, D^j_k]$ is $\sigma$-compatible with $[\Delta^0_k, d, D^0_k]$.
\item
$D^j_k$ dominates all faces of $\Delta^j_k$, and $d \notin D^j_k$.
\item
 $|D^j_k| - e({\Delta^j_k}[D^j_k]) = |D^0_k| - e({\Delta^0_k}[D^0_k]).$
\end{enumerate}
\end{lemma}

\begin{lemma}\label{lem3.4}
Let $L(\Delta^j_1)$ (or $R(\Delta^j_2)$), $1\leqslant j \leqslant 3$, be a configuration of $\Delta^j_1$ ($\Delta^j_2$, respectively) with the proper region bounded by the black cycle in Fig.~10. Suppose that $L(\Delta^0_3)$ (or $R(\Delta^0_3)$) is a  configuration of $\Delta^0_3$ with the proper region bounded by the black cycle on the left (on the right, respectively) in Fig.~11. Then
 $$[L(\Delta^j_1), d, D^j_1 \cap V(L(\Delta^j_1))] = [\Delta^0_1, d, D^0_1] = [R(\Delta^0_3), d, D^0_3 \cap V(R(\Delta^0_3))],$$
 and
      $$[R(\Delta^j_2), d, D^j_2 \cap V(R(\Delta^j_2))] = [\Delta^0_2, d, D^0_2] = [L(\Delta^0_3), d, D^0_3 \cap V(L(\Delta^0_3))].$$
\end{lemma}
\begin{lemma}\label{lem3.5}
For $k = 1,  2,  3$, $0 \leqslant  j \leqslant  n(k)$, where $n(1) = 3$, $n(2) = n(3) =6$, the following conditions are satisfied:
\begin{enumerate}[\upshape(1)]
\item
Every vertex of $D^j_k$, $k = 1, 2$, is linked with a vertex $a$ or $b$ by a path with all vertices belonging to $D^j_k$, where $a$ and $b$ are vertices in Fig.~10.

\item
 Every vertex of $D^j_3$ is linked with a vertex $a_1$, $b_1$ or $b_2$ by a path with all vertices belonging to $D^j_3$, where $a_1$, $b_1$ and $b_2$ are vertices in Fig.~11.

\item
The vertex $a_2 \in D^j_3$, for $j \neq 2$, is linked with  a vertex $b_1$ or $b_2$ by a path with all vertices belonging to $D^j_3$, where $a_2$, $b_1$, $b_2$ are vertices in Fig.~11.
\end{enumerate}
\end{lemma}

\begin{lemma}\label{lem3.6}
For $j = 4, 5, 6$, let $s_1\ldots s_{m(j)}$ be a path in $\Delta^j_3$ induced by the set $V^1(\Delta^j_3, d)$ such that $s_1 = a_1$ and $s_{m(j)} = a_2$, where $a_1$, $a_2$ are vertices in Fig.~11.

If $s_i \notin D^j_3$ and $d_{\Delta^j_3}(s_i) = 4$, then
 $s_{i-2} \in D^j_3$.

\end{lemma}

\begin{lemma}\label{lem3.7}
Suppose that $(G, q) \in \cal{G}$ has a hamiltonian set $U \subset V(G)$ which is compatible with the orientation of the cycle $c(G)$, and  $d_G(x) = 3$ for every vertex $x$ of $c(G)$ not belonging to $U$. Then $(G, g) = (G_3, g)$.
\end{lemma}

\begin{proof}
Let $(G, q) \neq (G_3, g)$ be a graph in $\mathcal{G}$ of minimum order satisfying the assumptions of Lemma \ref{lem3.7}. Since $U$ induces a tree in $G$, there is a vertex $x$ of $c(G)$ not belonging to $U$. Hence, $d_G(x) = 3$. Notice that $U$ is a hamiltonian set in the graph $(G - x, g)\in \mathcal{G}$ which satisfies the assumptions of  Lemma \ref{lem3.7}. Thus $(G - x, g) = (G_3, g)$, whence $(G, g) = (G_4, g)$. We obtain a contradiction, because  $(G_4, g)$ does not satisfy the assumptions of Lemma \ref{lem3.7}.
\qed \end{proof}

\begin{theorem}\label{theo3.1}
If $(G, g) \in {\mathcal{G}}$ has a hamiltonian set of vertices which is compatible with the orientation of $c(G)$, then every graph $(K,g) \in {\mathcal{G}}$ which is generated by operations of type $\bar{C}$, $\bar{B}$, $C$, $B$ from $(A^2(G), g)$ has a hamiltonian set of vertices which is $(-)$compatible with the orientation of $c(K)$.
\end{theorem}

\begin{proof}
Fix $(G, g) \in \mathcal{G}$. Suppose first that $U \subset V(G)$ is a hamiltonian set in $(G, g)$. Let us denote
 \begin{align*}
& N_1 = \{ 1\leqslant i \leqslant m: v_i \in U, v_{i+1} \notin U \},
\cr
& N_2 = \{ 1\leqslant i \leqslant m: v_i \in U, v_{i+1} \in U \},
\cr
& N_3 = \{ 1\leqslant i \leqslant m: v_i \notin U, v_{i+2} \in U \},
\cr
& N_4 = \{ 1\leqslant i \leqslant m: v_i \notin U, v_{i+2} \notin U \}.
\end{align*}

Suppose that $c(A^2(G))=  x_1 \ldots x_m$,  $c^2(A^2(G)) = y_1 \ldots y_m$.  Certainly, $G - g$ is is a configuration of $A^2(G)$. Hence, $c^3(A^2(G)) = v_1 \ldots v_m$ and 
$U \subset V(A^2(G))$.  We can assume, without loss of generality,  that $v_{i}y_{i}y_{i+1}$, $y_{i}x_{i}x_{i+1}$ are faces of the graph $(A^2(G), g)$ (see Fig.~12). Since $U$ induces a tree in $G$, $N_1 \neq \emptyset$. The set $U$ determines the following set $W \subset V(A^2(G))$
$$W = U \cup \{ y_{i+1}: i \in N_1 \} \cup
 \{ x_{i+1}: i \notin N_1 \}.$$


The definitions of the set $W$ gives us the following:
\begin{enumerate}
  \item[(1)] for $i \in N_1$, vertices of the path $v_iy_{i+1}x_{i+2}$ belong to $W$.
\end{enumerate}

Assume that $G_0$ is a subgraph of $A^2(G)$ induced by the set of vertices belonging to $c(A^2(G)) \cup c^2(A^2(G))$. Then, a subgraph of $G_0$ induced by $W \setminus U$ is the union of paths $P_{(k, l)} = y_{k+1}x_{k+2}x_{k+3}\ldots x_l$, where $k$, $l \in N_1$ (or it is the path $P_{(k, k)} = y_{k+1}x_{k+2}\ldots x_mx_1 \ldots x_k$, for $N_1 = \{k\}$). Since $k \in N_1$, condition~$(1)$ shows that $v_k$ is the only vertex belonging to $U$ which is adjacent to $P_{(k, l)}$ (or $P_{(k, k)}$, respectively). Hence, $W$ induces a tree in $A^2(G)$. Since $U$ dominates all faces of $G$, $W$ dominates all faces of $A^2(G)$. Thus, $W$ is a hamiltonian set in $(A^2(G),g)$.

Let $\Gamma_i$, $1\leqslant i \leqslant m$, be a configuration of $A^2(G)$ such that $x_ix_{i+1}g$, $x_ix_{i+1}y_i$, $y_iy_{i+1}x_{i+1}$, $y_iy_{i+1}v_i$ are the only proper faces of $\Gamma_i$. By the definitions of the set $W$ and the patterns $[\Delta^0_k, d, D^0_k]$, $k = 1, 2$, it follows that
\begin{enumerate}
 \item[(2)]
 $\left\{
 \begin{array}{l}
[\Gamma_i, g, W\cap V(\Gamma_i)] = [\Delta^0_1, d, D^0_1], \hbox{ for } i \in N_1,
\\[4pt]
[\Gamma_i, g, W\cap V(\Gamma_i)] = [\Delta^0_2, d, D^0_2], \hbox{ for } i \in N_2,
 \end{array}
 \right.$
 \end{enumerate}

Let $\Gamma(G, U)$ be the set of all functions $\bar{\gamma}: N_1 \cup N_2 \to \{ 0,1,2,3 \}$ satisfying the following conditions:

 $\bar{\gamma} (i) \leqslant 3$, for $d_{A^2(G)}(v_i) = 4$,

 $\bar{\gamma} (i) \leqslant 1$, for $d_{A^2(G)}(v_i) = 5$.

 $\bar{\gamma} (i) = 0$, for $d_{A^2(G)}(v_i) = 6$.
\\
Fix $\bar{\gamma} \in \Gamma(G, U)$.
Certainly, proper regions of $\Gamma_i$, $1\leqslant i \leqslant m$, are pairwise disjoint. Hence, by (2) and Lemma \ref{lem3.3}(1), we can define a replacement operation
 $$\gamma: [\Gamma_i, g,  W\cap V(\Gamma_i)]
 \ {\xrightarrow[]{i \in N_k, \ k = 1, 2}} \
 [ \Delta^{\bar{\gamma}(i)}_k, d, D^{\bar{\gamma}(i)}_k],$$
to obtain a plane triangulation $(\gamma(A^2(G)),g)$ from $(A^2(G),g)$, and the following set
 $$\gamma(W)= (W \setminus \bigcup_{i \in N_1 \cup N_2} V(\Gamma_i)) \cup  \bigcup_{i \in N_1}D^{\bar{\gamma}(i)}_1 \cup \bigcup_{i \in N_2} D^{\bar{\gamma}(i)}_2 \subset V(\gamma(A^2(G))).$$
Notice that $c^3(\gamma(A^2(G))) = c^3(A^2(G)) = v_1 \ldots v_m$. By the definition of the function $\bar{\gamma}$, $d_{\gamma(A^2(G))}(v_i) \leqslant 6$.
Hence, $(\gamma(A^2(G)), g) \in {\mathcal{G}}$.

Recall that $W$ induces a tree in $A^2(G)$. Therefore, by  Lemma~\ref{lem3.3}(1) and Lemma~\ref{lem3.5}(1), $\gamma(W)$ induces a connected subgraph of $\gamma(A^2(G))$. Since the tree $A^2(G)[W]$ has $|W|- 1$ edges, by  Lemma~\ref{lem3.3}(3) and Lemma~\ref{lem3.3}(1), the induced graph $\gamma(A^2(G))[\gamma(W)]$ has
\begin{align*}
e(A^2(G)[W])
&+ \sum_{i \in N_1}\{e({\Delta^{\bar{\gamma}(i)}_1}[D^{\bar{\gamma}(i)}_1]) - e({\Delta^0_1}[D^0_1])\}
\cr
&+ \sum_{i \in N_2}\{e({\Delta^{\bar{\gamma}(i)}_2}[D^{\bar{\gamma}(i)}_2]) - e({\Delta^0_2}[D^0_2])\}
\cr
&= |W| - 1 + \sum_{i \in N_1} \{|D^{\bar{\gamma}(i)}_1|- |D^0_1|\}
\cr
&+ \sum_{i \in N_2} \{|D^{\bar{\gamma}(i)}_2| - |D^0_2|\}
 = |\gamma(W)| - 1
\end{align*}
edges. Hence, $\gamma(W)$ induces a tree in $\gamma(A^2(G))$.

Let us consider a graph $(G, g)$ such that $c(G)$ has length $3$. If every vertex of $c(G)$ has degree $5$ in $G$, then $(A^2(G), g)$ is the only graph which is generated by operations of type $\bar{C}$, $\bar{B}$ from $(A^2(G), g)$. If  $c(G) = v_1v_2v_3$ has a vertex of degree at most $4$, then $(G, g) = (G_3, g)$, or $(G, g) = (F_1, g)$, or $(G, g) = (F_2, g)$, with the hamiltonian set $U = \{v_1, v_2\}$ marked by circles in Fig. ~4, or Fig.~5, respectively.  Notice that $d_{A^2(G)}(v_2) \leqslant 5$ and $2 \in N_1$. If  $(G, g) = (G_3, g)$, or $(G, g) = (F_1, g)$, then vertices $v_1$, $v_2$, $v_3$ are similar in $A^2(G) - g$. Hence,  we can assume  the following:
\begin{enumerate}
\item[(3)]
if  $(G, g) = (G_3, g)$, or $(G, g) = (F_1, g)$, then $\bar{\gamma} (v_2) \geqslant 1$.\end{enumerate}
Since $d_{A^{2}(F_2)}(v_1) = d_{A^{2}(F_2)}(v_3) = 6$, we obtain the following:
\begin{enumerate}
\item[(4)]
every graph of $\mathcal{G}$ which is generated  by operations of type $\bar{C}$ and  $\bar{B}$ from $(A^2(F_2), g)$ is op-equivalent to $(\delta(A^2(F_2)), g)$, for some $\bar{\delta} \in \Gamma(G, U)$.
\end{enumerate}

Since proper regions of configurations $\Gamma_i$ of $A^2(G)$,  $1\leqslant i \leqslant m$, are pairwise disjoint, we can assume the following: 
\begin{enumerate}
 \item[(5)]
 every configuration   $\Gamma_i$ of $A^2(G)$,  for $i \in N_3 \cup N_4$, is a configuration of the graph $\gamma(A^2(G))$.
\end{enumerate}
Thus, vertices $x_i$, $x_{i+1}$, $y_i$, $y_{i+1}$ of $\Gamma_i$, for  $i \in N_3$,  are also verices of $\gamma(A^2(G))$. Hence, we can define a minimal configuration $\Lambda_i$ of $\gamma(A^2(G))$ such that $x_i$, $x_{i+1}$, $y_i$, $y_{i+1}$ are the only vertices in the proper region of $\Lambda_i$. If $c(G)$ has length~$3$, then, by condition (3) and (4), we can assume that $\bar{\gamma} \neq 0$. Hence, cycles $c(\gamma(A^2(G))$, $c^2(\gamma(A^2(G))$ have lengths at least $4$, and the configuration $\Lambda_i$ of $\gamma(A^2(G))$ do exists. If $i \in N_3$, then $i-1 \in N_1$, and $i+1 \in N_2$. Hence, by the definition of $\Lambda_i$, the definition of $\gamma$, and by Lemma \ref{lem3.4} we obtain:
\begin{enumerate}
\item[(6)]
$[\Lambda_i, g, \gamma(W)\cap V(\Lambda_i)] = [\Delta^0_3, d, D^0_3], \hbox{ for } i \in N_3$.
\end{enumerate}

Let $\Lambda(G, U)$ be the set of all functions  $\bar{\lambda}: N_3 \to \{ 0,1, \ldots, 6 \}$ satisfying the following conditions:

 $\bar{\lambda}(i) \leqslant 6$, for $d_{A^2(G)} (v_i) = 4$,

 $\bar{\lambda}(i)\leqslant 1$, for  $d_{A^2(G)} (v_i) = 5$,

 $\bar{\lambda}(i) = 0$, for  $d_{A^2(G)} (v_i) = 6$,

 $\bar{\lambda}(i)\neq 2$, for $(G, g) = (G_3, g)$.
\\
Fix $\bar{\lambda} \in \Lambda(G, U)$. If $i \in N_3$, then $i-2 \notin N_3$, and $i+2 \notin N_3$. Hence, proper regions of configurations $\Lambda_i$, for $i \in N_3$, are pairwise disjoint. Therfore, by (6) and by Lemma \ref{lem3.3}(1), we can define a replacement operation
$$\lambda: [\Lambda_i, g, \gamma(W) \cap V(\Lambda_i)]
 \ {\xrightarrow[]{i \in N_3}} \
 [ \Delta^{\bar{\lambda}(i)}_3, d, D^{\bar{\lambda}(i)}_3],$$
to obtain a plane triangulation $(\lambda \circ \gamma(A^2(G)),g)$ from $(\gamma(A^2(G)),g)$, and the following set
 $${\lambda \circ \gamma(W)}= (\gamma(W) \setminus \bigcup_{i \in N_3} V(\Lambda_i)) \cup  \bigcup_{i \in N_3}D^{\bar{\lambda}(i)}_3  \subset V(\lambda \circ \gamma(A^2(G))).$$
If $N_3 = \emptyset$ we put $(\lambda \circ \gamma(A^2(G)),g) = (\gamma(A^2(G)),g)$ and ${\lambda \circ \gamma(W)} = \gamma(W)$.

Notice that $c^3(\lambda \circ \gamma(A^2(G)) = c^3(\gamma(A^2(G))) = v_1 \ldots v_m$. By the definition of $\bar{\lambda}$, $d_{\lambda \circ \gamma(A^2(G))}(v_i) \leqslant 6$, for $i \in N_3$. Hence, $(\lambda \circ \gamma(A^2(G)),g) \in {\mathcal{G}}$.

Assume now that the set $U$ is compatible with the orientation of $c(G)$. We will prove the following:
\begin{enumerate}
 \item[(7)]
$\lambda \circ \gamma(W)$ is a hamiltonian set in $(\lambda \circ \gamma(A^2(G)),g)$ which is $(-)$compatible with the orientation of the cycle $c(\lambda \circ \gamma(A^2(G)))$.
\end{enumerate}
First, we prove that $\lambda \circ \gamma(W)$ induces a connected subgraph of  $\lambda \circ \gamma(A^2(G))$. Since $\gamma(W)$ induces a tree in $\gamma(A^2(G))$, by Lemma \ref{lem3.3}(1) and Lemma \ref{lem3.5}(2), it is sufficient to prove the following:
\begin{enumerate}
\item[(8)]
cycles $c(\lambda \circ \gamma(A^2(G)))$ and $c^3(\lambda \circ \gamma(A^2(G))= v_1 \ldots v_m$ are linked by a path with all vertices  belonging to $\lambda \circ \gamma(W)$.
\end{enumerate}
Remark that if $\bar{\lambda}(i)\neq 2$, for some $i \in N_3$, then,  by Lemma \ref{lem3.5}(3), condition (8) holds. Thus, for the graph $(G, g) = (G_3, g)$, condition (8) is satisfied.

If $(G, g) \neq (G_3, g)$, then, by Lemma \ref{lem3.7}, there exists a vertex $v_l \notin U$ such that  $d_{A^2(G)} (v_l) \geqslant  5$. If $l \in N_3$, then $\bar{\lambda}(l)\leqslant 1$. Thus, condition (8) holds.  If $l \in N_4$, then  $l-1 \in N_1$. Thus, by conditions (1), (5) and Lemma \ref{lem3.3}(1), vertices of the path $v_{l-1}y_lx_{l+1}$ belong to $\lambda \circ \gamma(W)$. Hence, condition (8) holds.

Since the tree $\gamma(A^2(G))[\gamma(W)]$ has $|\gamma(W)| - 1$ edges, by Lemma~\ref{lem3.3}(3) and Lemma \ref{lem3.3}(1), the induced graph $\lambda \circ \gamma(A^2(G))[\lambda \circ \gamma(W)]$ has
\begin{align*}
e(\gamma(A^2(G))[\gamma(W)]) + \sum_{i \in N_3}\{e({\Delta^{\bar{\lambda}(i)}_3}[D^{\bar{\lambda}(i)}_3]) - e({\Delta^0_3}[D^0_3])\}
\cr
= |\gamma(W)| - 1 + \sum_{i \in N_3} \{|D^{\bar{\lambda}(i)}_3|- |D^0_3|\}
 = |\lambda \circ \gamma(W)| - 1
\end{align*}
edges. Hence, $\lambda \circ \gamma(W)$ induces a tree in  $\lambda \circ \gamma(A^2(G))$.

Since $W$ dominates all faces in $A^2(G)$, by Lemma \ref{lem3.3}, $\gamma(W)$ dominates all faces in $\gamma(A^2(G))$ and
$\lambda \circ \gamma(W)$  dominates all faces in $\lambda \circ \gamma(A^2(G))$. Furthermore, every vertex of the cycle $c(\gamma(A^2(G)))$ has degree $5$ in $\gamma(A^2(G))$. Thus, $\gamma(W)$ is  $(\pm)$compatible with the orientation of $c(\gamma(A^2(G)))$. Therefore, $\lambda \circ \gamma(W)$ is $(-)$compatible with the orientation of the cycle $c(\lambda \circ \gamma(A^2(G)))$, by Lemma \ref{lem3.6}. Hence, (7) holds.

If $i \in N_4$, then $d_{A^2(G)} (v_i) = d_G (v_i) + 1 = 6$, because $U$ is compatible with the orientation of the cycle $c(G)$. Thus,
 if $d_{A^2(G)} (v_i) \leqslant 5$, then $i \in {N_1 \cup N_2 \cup N_3}$.
Hence, by the definitions of $\gamma$ and $\lambda$, it is not difficult to check the following:

\begin{enumerate}
\item[(9)]
if $(G, g) \neq (G_3, g)$, then every graph of $\mathcal{G}$ which is generated  by operations of type $\bar{C}$ and  $\bar{B}$ from $(A^2(G), g)$ is op-equivalent to $(\lambda \circ \gamma(A^2(G)), g)$, for some $\bar{\gamma} \in \Gamma(G, U)$ and $\bar{\lambda} \in \Lambda(G, U)$.
\end{enumerate}

Let $(R, g) \in \cal{G}$ be a graph defined as the graph $R - g \in \cal{F}$ shown in Fig.~13. Notice that it is generated by operations of type $\bar{C}$ from  $(A^2(G_3), g)$. Suppose that $c^{3}(A^2(G_3)) =  c(G_3) = v_1v_2v_3$ and $U = \{v_1, v_2\}$ is a hamiltonian set in $(G_3, g)$. Hence, $3 \in N_3$ and $\bar{\lambda}(3)\neq 2$. Thus, the graph $(R, g)$ is not op-equivalent to $(\lambda \circ \gamma(A^2(G_3), g)$, for any $\bar{\gamma} \in \Gamma(G_3, U)$ and $\bar{\lambda} \in \Lambda(G_3, U)$. Since vertices $v_1$, $v_2$ and $v_3$ are similar in $A^2(G_3) - g$, $1 \in N_2$ and $2 \in N_1$, it is easy to check the following:
\begin{enumerate}
\item[(10)]
every graph of  $\mathcal{G}$, different than $ (R, g)$, which is generated by operations of type $\bar{C}$ and  $\bar{B}$ from  $(A^2(G_3), g)$ is op-equivalent to $(\lambda \circ \gamma(A^2(G_3)), g)$, for some $\bar{\gamma} \in \Gamma(G_3, U)$ and $\bar{\lambda} \in \Lambda(G_3, U)$.
\end{enumerate}
Certainly, $(R, g)$ has a hamiltonian set of vertices, marked by circles in Fig.~13, which is ($\pm$)compatible with the orientation of the cycle $c(R)$.

Fix $(H_0, g)= (\lambda_0 \circ \gamma_0(A^2(G)), g)$, and suppose that $X_0 = \lambda_0 \circ \gamma_0(W)$. Let $c^2(H_0) = t_1 \ldots t_n$. As in Definition \ref{def3.2}, we define sets $M_1, \ldots, M_4$, for the graph $(H_0, g)$ and the set $X_0$. Similarly, we define the set $\Omega(H_0, X_0)$ and the replacement operation $\omega$, for $\bar{\omega} \in \Omega(H_0, X_0)$.

Let us consider a vertex $t_m \in c^2(H_0)$ of degree $d_{H_0}(t_m) \leqslant 5$.  By the definition of $\lambda_0$ and  $\gamma_0$ the vertex $t_m$ belongs to a configuration $\Gamma$ of $(H_0, g)$ such that 
$$ [\Gamma, g, X_0 \cap V(\Gamma)] = [\Delta^{\bar{\gamma}_0(l)}_k, d, D^{\bar{\gamma}_0(l)}_k], \hbox{ for some } l = l(m) \in N_k, \ k = 1, 2,$$
or 
$$ [\Gamma, g, X_0 \cap V(\Gamma)] = [\Delta^{\bar{\lambda}_0(l)}_3, d, D^{\bar{\lambda}_0(l)}_3], \hbox{ for some } l = l(m) \in N_3.$$
By the definition of patterns $[\Delta^j_k, d, D^j_k]$, we obtain:
\begin{align*}
\hbox{if } &l \in N_1 \cup N_2, \hbox{ then } m \in M_1 \cup M_2,
\cr
\hbox{if } &l \in N_3 \hbox{ and } \bar{\lambda}_0(l) \neq 2, \ \bar{\lambda}_0(l) \neq 3, \hbox{ then } m \in M_1 \cup M_2 \cup M_3,
\cr
\hbox{if } &l \in N_3 \hbox{ and } \bar{\lambda}_0(l) = 2 , \hbox{ then } (C_{t_m}(H_0), g) = (\lambda \circ \gamma_0(A^2(G)), g), \hbox{ where }
\cr
&\bar{\lambda} \in \Lambda(G, U) \hbox{ is a function: } \bar{\lambda}(l) = 4 \hbox{ or } 5  \hbox{ and } \bar{\lambda}(i) = \bar{\lambda}_0(i),  \hbox{ for }  i \in  N_3  \setminus \{l\}.
\cr
\hbox{if } &l \in N_3 \hbox{ and } \bar{\lambda}_0(l) = 3, \hbox{ then } (C_{t_m}(H_0), g) = (\lambda \circ \gamma_0(A^2(G)), g), \hbox{ where }
\cr
&\bar{\lambda} \in \Lambda(G, U) \hbox{ is the function: } \bar{\lambda}(l) = 6  \hbox{ and } \bar{\lambda}(i) = \bar{\lambda}_0(i), \hbox{ for }  i \in  N_3  \setminus \{l\}.
\end{align*}
Hence, by $(9)$--$(10)$ and the definition of $\omega$, we obtain the following:
\begin{enumerate}
\item[(11)]
every graph of $\mathcal{G}$, different than $(R, g)$, which is generated  by operations of type $\bar{C}$, $\bar{B}$, $C$, $B$ from $(A^2(G), g)$ is op-equivalent to one of the following graphs:
$$(H, g ) = (\lambda \circ \gamma(A^2(G)), g), \hbox{ for some }\bar{\gamma} \in \Gamma(G, U),  \bar{\lambda} \in \Lambda(G, U),$$
or
$$(\omega(H), g), \hbox{ for some } \bar{\omega} \in \Omega(H, X), \hbox{ where } X = \lambda \circ \gamma(W).$$
\end{enumerate}
By (7), $X = \lambda \circ \gamma(W)$ is a hamiltonian set in $(H, g)= (\lambda \circ \gamma(A^2(G)), g) $ which is $(-)$compatible with the orientation of $c(H)$. Hence, by Lemma \ref{lem3.2}, the graph $(\omega(H), g)$ has a hamiltonian set $\omega(X)$ which is $(-)$compatible with the orientation of $c(\omega(H))$. In view of $(11)$, the proof is completed.
\qed \end{proof}

\begin{lemma}\label{lem3.8}
Every graph $(K,g) \in {\mathcal{G}}$ which is generated by operations of type $\bar{C}$, $\bar{B}$, $C$, $B$ from $(A^2(P), g)$ (or $(A^2(Q), g)$) has a hamiltonian set of vertices which is $(\pm)$compatible with the orientation of $c(K)$.
\end{lemma}
\begin{proof}
Let $(D, g)$ and $(E, g) \in \cal{G}$ be graphs defined by the graphs $D - g$ and $E - g  \in \cal{F}$ in Fig.~14. These graphs have hamiltonian sets $X_D \subset V(D)$ and $X_E \subset V(E)$, marked by circles in in Fig.~14, which are $(\pm)$compatible with the orientation of the cycles $c(D)$ and $c(E)$, respectively. Notice that $(D, g)$ is generated by operations of type $\bar{C}$ from $(A^2(P), g)$, and $(E, g)$  by operations of type $C$ from $(D, g)$.

It is sufficient to prove the following two conditions:
\begin{enumerate}
\item [(i)] Every graph $(K,g) \in {\mathcal{G}}$,  $(K, g)\neq (D, g)$, which is generated by operations of type $\bar{C}$, $\bar{B}$, $C$, $B$ from $(A^2(P), g)$ has a hamiltonian set of vertices which is $(\pm)$compatible with the orientation of $c(K)$.
\end{enumerate}
\begin{enumerate}
\item [(ii)]  Every graph $(K,g) \in {\mathcal{G}}$,  $(K, g)\neq (E, g)$, which is generated by operations of type  $C$, $B$ from $(D, g)$ has a hamiltonian set of vertices which is $(\pm)$compatible with the orientation of $c(K)$.

 \end{enumerate}
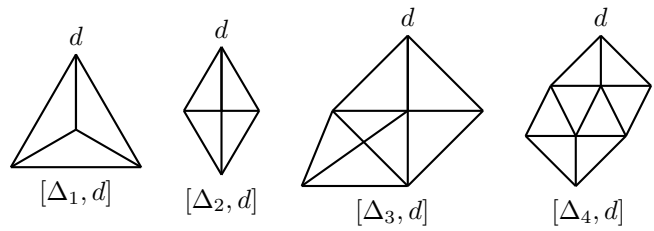
Suppose that $U$ is a hamiltonian set in $(P, g)$ marked by circles in Fig.~2.  Let us denote $c(P) = v_1 \ldots  v_6$, where vertices $v_1$, $v_3 \notin U$ and  other vertices of $c(P)$ belong to $U$. We define the set $N_k$, for $k = 1, 2, 3, 4$, and the set $W \subset V(A^2(P))$ determined by $U$, as in the proof of Theorem \ref{theo3.1}.  Notice that $W$ is a hamiltonian set of vertices in $(A^2(P), g)$ marked by circles in Fig.~14. We define $\Gamma(P, U)$,  $\Lambda(P, U)$, replacement operations $\gamma$ and $\lambda$, for $\bar{\gamma} \in \Gamma(P, U)$ and $\bar{\lambda} \in \Lambda(P, U)$, and the plane triangulation $(\lambda \circ \gamma(A^2(P)), g)$ with the set $\lambda \circ \gamma(W)$, as in the proof of Theorem \ref{theo3.1}.
\vfill


Remark that vertices $v_1$, $v_3$, $v_5$ are of degree $5$ in $A^2(P)$, and the remaining vertices of $c^3(A^2(P))$ are of degree $6$. Since $5 \in N_2$, $3 \in N_3$,  $\bar{\gamma}(5)\leqslant 1$  and $\bar{\lambda}(3) \leqslant 1$,  every vertex of $c(\lambda \circ \gamma(A^2(P)), g)$ has degree $5$.  Furthermore, since $\bar{\lambda}(3)~\neq~
 2$, by Lemma \ref{lem3.5}, the set  $\lambda \circ \gamma(W)$ induces a connected subgraph of $\lambda \circ \gamma(A^2(P))$. Hence, by Lemma \ref{lem3.3}, it is a hamiltonian set in $(\lambda \circ \gamma(A^2(P)), g)$ which  is $(\pm)$compatible with the orientation of $c(\lambda \circ \gamma(A^2(P)))$.

Since $v_1$ is of degree $5$ in $A^2(P)$ and $1 \in N_4$, $(D, g)$ is not op-equivalent to $(\lambda \circ \gamma(A^2(P)), g)$, for any function $\bar{\gamma} \in \Gamma(P, U)$ and $\bar{\lambda} \in \Lambda(P, U)$. Notice that vertices  $v_1$, $v_3$, $v_5$ are similar in $A^2(P) - g$. Hence, every graph of $\mathcal{G}$, different than $(D, g)$, which is generated by operations of type  $\bar{C}$, $\bar{B}$  from $(A^2(P), g)$ is op-equivalent to $(\lambda \circ \gamma(A^2(P)), g)$, for some $\bar{\gamma} \in \Gamma(P, U)$   and  $\bar{\lambda} \in \Lambda(P, U)$.

Fix  $(H, g)= (\lambda_0 \circ \gamma_0(A^2(G)), g)$, and suppose that $X = \lambda_0 \circ \gamma_0(W)$. Let $c^2(H) = t_1 \ldots t_n$. As in Definition \ref{def3.2}, we define sets $M_1, \ldots, M_4$, for the graph $(H,  g)$ and the set $X$. Similarly, we define the set $\Omega^*(H, X)$ and the replacement operation $\omega$, for $\bar{\omega} \in \Omega^*(H, X)$.

Let us consider a vertex $t_m \in c^2(H)$ of degree $d_{H}(t_m) \leqslant 5$.  Recall that  $v_5$,~$v_3$ are of degree $5$ in $A^2(P)$, $5 \in N_2$ and $3 \in N_3$.  Hence, by the definition of  $\gamma_0$ and $\lambda_0$ the vertex $t_m$ belongs to a configuration $\Gamma$ of $(H, g)$ such that 
$$ [\Gamma, g, X \cap V(\Gamma)] = [\Delta^{\bar{\gamma}_0(5)}_2, d, D^{\bar{\gamma}_0(5)}_2] = [\Delta^1_2, d, D^1_2],$$ 
or
$$[\Gamma, g, X \cap V(\Gamma)] =  [\Delta^{\bar{\lambda}_0(3)}_3, d, D^{\bar{\lambda}_0(3)}_3] = [\Delta^1_3, d, D^1_3].$$ 
Hence, $m \in M_1 \cup M_2$ and $d_{H}(t_m) = 5$. Thus, by the definition of $\omega$, every graph $(K, g) \in \mathcal{G}$ which is generated by operations of type $C$, $B$ from $(H, g)$ is op-equivalent to $(\omega(H), g)$,  for some $\bar{\omega} \in \Omega^*(H, X)$. By Lemma \ref{lem3.2}, $(K, g)$ has a hamiltonian set which is $(\pm)$compatible with the orientation of $c(K)$. Hence, (i) holds.

Suppose that $X_D$  is a hamiltonian set of vertices in $(D, g)$ marked by circles in  Fig.~14. Let us denote $c^2(D) = t_1t_2 \ldots t_9$, where vertices $t_1$, $t_6$, $t_9 \in X_D$ and  other vertices of $c^2(D)$ don't belong to $X_D$.  Let us define sets $M_1, \ldots, M_4$, the set $\Omega^*(D, X_D)$ and the replacement operation $\omega$, for $\bar{\omega} \in \Omega^*(D, X_D)$, as in Definition \ref{def3.2}.

Remark that $t_1$, $t_4$, $t_7$ are of degree $5$  in $D$, and the remaining vertices of $c^2(D)$ are of degree $6$. Since $t_1 \in X_D$, $(E, g)$ is not op-equivalent to $(\omega(D), g)$,  for ${\bar{\omega} \in \Omega^*(D, X_D)}$. Notice that vertices  $t_1$, $t_4$, $t_7$  are similar in $D - g$, and ${4, 7 \in M_2}$. Hence, every graph $(K, g) \in \mathcal{G}$, $(K, g)\neq (E, g)$, which is generated by operations of type $C$, $B$ from $(D, g)$ is op-equivalent to $(\omega(D), g)$,  for some $\bar{\omega} \in \Omega^*(D, X_D)$. By Lemma~\ref{lem3.2}, $(K, g)$ has a hamiltonian set which is $(\pm)$compatible with the orientation of $c(K)$. Hence, (ii) holds.
\qed \end{proof}

\begin{lemma}\label{lem3.9}
Let $(F,g) \in \cal {G}$, and suppose that $(F, g) = (r(F), g)$, where $r(G)$ is a mirror reflection of $F$.  Let ${\cal {F}}_m$, $m \geqslant 2$, be a family of all graphs of height $m$, different than $(P, g)$ and $(Q, g)$, which are generated by operations of type $A$, $B$, $C$ from $(F, g)$. If every graph $(K,g) \in {\cal {F}}_m$ has a hamiltonian set which is ($-$)compatible with the orientation of $c(K)$, then every graph $(G, g) \in {\cal {F}}_m$ has a hamiltonian set which is compatible with the orientation of $c(G)$.
\end{lemma}
\begin{proof}
Let $(G, g) \in {\cal {F}}_m$. Notice that $(r(G), g)$ is generated by operations of type $A$, $B$, $C$ from $(r(F), g) = (F, g)$. Furthermore, $h(r(G)) = m$, and $(P, g) \neq (r(G), g) \neq (Q, g)$. Hence, $(r(G), g) \in {\cal {F}}_m$ and it has a hamiltonian set which is ($-$)compatible with the orientation of $c(r(G))$. Thus, $(G, g)$ has a hamiltonian set  which is compatible with the orientation of $c(G)$. 
\qed \end{proof}

In the proof of Theorem \ref{theo3.2} given below we shall use Theorem \ref{theo4.1} which will be proved in the next section.

\begin{theorem}\label{theo3.2}
If $(G,g) \in \cal {G}$ is different than $(P, g)$ and $(Q, g)$, then it has a hamiltonian set of vertices which is compatible with the orientation of $c(G)$.
\end{theorem}

\begin{proof}
The proof is by induction with respect to  height $h(G)$. Let ${(G, g) \in {\cal {G}}_n}$ be any graph of height $m$, which is different than $(P, g)$ and $(Q, g)$. By Theorem~\ref{theo4.1}, we can assume that $m > 3$, for $n > 0$ ($m > 2$, for $n = 0$). Hence, by  Theorem \ref{theo2.2}, $(G^{-2}, g) \in {\cal {G}}_n$ and $(G, g)$ is generated  by operations $\bar{C}$, $\bar{B}$, $C$, $B$ from $(A^2(G^{-2}), g)$. By induction, if $(G^{-2}, g)$ is different than $(P, g)$ and $(Q, g)$, then it  has a hamiltonian set which is compatible with the orientation of $c(G^{-2})$. Thus, by Theorem \ref{theo3.1} and Lemma \ref{lem3.8}, $(G, g)$ has a hamiltonian set which is ($-$)compatible with the orientation of $c(G)$. Hence, by Lemma \ref{lem3.9}, $(G, g)$ has a hamiltonian set which is compatible with the orientation of $c(G)$, because $(r(G_3), g) = (G_3, g)$, and $(r(F_n), g) = (F_n, g)$, for $n \geqslant 1$.
\qed \end{proof}

\section{Hamiltonian sets in graphs of small heights}
In Theorem \ref{theo4.1},  we shall prove that every graph $(G, g) \in \cal {G}$ of small height which is different from $(P, g)$ and from $(Q, g)$ has a hamiltonian set of vertices which is compatible with the orientation of $c(G)$. We will divide the proof of Theorem \ref{theo4.1} into a sequence of lemmas.

Recall that  $\{ {\cal{G}}_0, {\cal{G}}_1, {\cal{G}}_2, \ldots \}$ is a partition of $\cal{G}$ (see  Corollary \ref{coro2.3}).

\begin{lemma}\label{lem4.1}
Every graph $(G, g) \in {\mathcal{G}}_0$ of height $1$ has a hamiltonian set of vertices which is $(\pm)$compatible with the orientation of $c(G)$.
\end{lemma}
\begin{proof}
Lemma \ref{lem4.1} follows from  Corollary \ref{coro2.1}(1), because $(J, g)$ and every graph $(G_n, g)$, for $n \geqslant 3$, has a hamiltonian set of vertices, marked
by circles in Fig.~4, which is $(\pm)$compatible with the orientation of $c(J)$ and $c(G_n)$, respectively.
\qed \end{proof}

\begin{lemma}\label{lemma4.2}
Every graph $(G, g) \in  {\cal{G}}_0$, $(G, g) \neq (P, g)$, which is generated by operations of type $C$ and $B$ from one of the graphs: $(A(G_3), g)$, $(A(G_4), g)$, $(A(G_5), g)$, or $(A(J), g)$, has a hamiltonian set of vertices which is compatible with the orientation of $c(G)$.
\end{lemma}
\begin{proof}
Let $(A_j, g) \in  \cal{G}$ be a graph defined by the graph $A_j - g \in  \cal{F}$ in Fig.~15, where $1 \leqslant  j \leqslant 8$. Certainly,
\begin{align*}
(A_1, g) &= (A(G_3), g), (A_5, g) = (A(G_4), g),
\cr
(A_7, g) &= (A(G_5), g), (A_8, g) = (A(J), g).
\end{align*}
Every graph $(A_j, g)$, where $1 \leqslant  j \leqslant 8$, has a hamiltonian set of vertices, say $X_j$, marked by circles in Fig.~15. This set is ($\pm$)compatible for $j \neq 4$ (or, ($-$)compatible for $j = 4$) with the orientation of the cycle $c(A_j)$.





 As in Definition \ref{def3.2} we define the set of functions $\Omega(A_j, X_j)$ and the plane triangulation $(\omega(A_j), g)$, for $\bar{\omega} \in \Omega(A_j, X_j)$ and 
$j \neq 4$.

Let $(G, g)$  be a graph which is generated by operations of type $C$ and $B$ from  $(A_j, g)$, for some  $j = 1, 5, 7, 8$. It is easy to check that if $(G, g)$ is generated  from $(A_1, g)$, then, by the  definition of $\omega$, it is op-equivalent to one of the following graphs:
\begin{align*}
&(\omega(A_j), g), \hbox{ for }  j = 1, 2, 3\hbox{ and } \bar{\omega} \in \Omega(A_j, X_j),
\cr
&(A_4, g), \hbox{ or } (P, g).
\end{align*}
By analogy, we check that if $(G, g)$ is generated from  $(A_j, g)$, for $j = 5, 7, 8$, then it is op-equivalent to one of the following graphs:
\begin{align*}
(\omega(A_j), g), \hbox{ for } j = 5, 6, 7, 8,  \hbox{ and } \bar{\omega} \in \Omega(A_j, X_j).
\end{align*}
Hence, by Lemma \ref{lem3.2},  $(G, g)  \neq (P, g)$ has a hamiltonian set of vertices which is ($-$)compatible with the orientation of $c(G)$. Notice that
$$
(r(A(G_j)), g) = (A(G_j), g), \hbox{ for } j = 3, 4, 5, \hbox{ and } (r(A(J)), g)= (A(J), g).
$$
Thus, by Lemma~\ref{lem3.9},  the graph $(G, g)$, different than  $(P, g)$, has a hamiltonian set which is compatible with the orientation of $c(G)$.
\qed \end{proof}

\begin{definition}\label{def4.1}
 By a \textsl{snake path} of the graph $(G_n, g) \in  {\cal{G}}_0$, for $n\geqslant 5$, we will mean an induced path in $G_n - g$, say $\mathbb{W}_n$,  such that ${c(G_n) - \mathbb{W}_n}$ is the union of pairwise disjoint edges (see Fig.~16 and Fig.~17).


The graph $G_n - g$ has exactly two \textsl{distinguish edges}, say $u_1 v_1$, $u_2 v_2$, such that  vertices $u_1$, $u_2$ (or $v_1$, $v_2$) have degree $2$ (or $3$, respectively).  Assume that $u_1v_1$ is the first edge on the left side, and $u_2 v_2$ is the first edge on the right side of $G_n - g$ (see Fig.~16 and Fig.~17). An integer $\alpha_i (\mathbb{W}_n)$, $i = 1, 2$, is assigned to every snake path $\mathbb{W}_n$ of $(G_n, g)$ by:
$$
\alpha_i (\mathbb{W}_n) = \left\{
\begin{array}{lll}
1 \ &\hbox{for} \ & u_i,v_i \notin \mathbb{W}_n,
\cr
2 \  &\hbox{for}\ & u_i \notin \mathbb{W}_n, v_i \in \mathbb{W}_n,
\cr
3 \  &\hbox{for}\ & u_i \in \mathbb{W}_n, v_i \notin \mathbb{W}_n,
\cr
4 \ &\hbox{for} \ & u_i,v_i \in \mathbb{W}_n.
\end{array}
\right.
$$
\end{definition}

\begin{lemma}\label{lem4.3}
 Every graph $(G_n, g) \in {\cal{G}}_0$, $n\geqslant 5$, has snake paths $\mathbb{W}_n$ and $\mathbb{Z}_n$  such that $\alpha_i(\mathbb{W}_n)$ and $\alpha_i(\mathbb{Z}_n): \{n \in \mathbb{N}: n \geqslant 5 \} \to \{1, 2, 3, 4 \}$, for $i = 1$, $2$, are  periodic functions (with period $5$) satisfying the following conditions:
$$
\begin{array}{lllll}
\alpha_1(\mathbb{W}_n) = j, &\hbox{and} &\alpha_2(\mathbb{W}_n) = 1,
 &\hbox{for}
& n \equiv j \pmod5,\ j = 1, 2, 3, 4,
\cr
\alpha_1(\mathbb{W}_n) = 4, &\hbox{and} &\alpha_2(\mathbb{W}_n) = 2,
 &\hbox{for}
& n \equiv 0 \pmod5,
 \cr
 \alpha_1 (\mathbb{Z}_n) = j+2, &\hbox{and} & \alpha_2(\mathbb{Z}_n) = 4,
&\hbox{for} & n \equiv j \pmod5,\  j = -1, 0, 1, 2,
\cr
 \alpha_1 (\mathbb{Z}_n) = 1, &\hbox{and} & \alpha_2(\mathbb{Z}_n) = 3,
&\hbox{for} & n \equiv 3 \pmod5.
\end{array}
$$
\end{lemma}

\begin{proof}
Let $u_1 v_1$,  $u_2 v_2$ be two distinguished edges of   $G_n -g$, for $n \geqslant 5$. 
A snake path  $\mathbb{W}_n$ (or $\mathbb{Z}_n$) of  $(G_n, g)$, where $n = k + 5j$, $k = 5, \ldots, 9$ and $j = 0, 1$, is indicated by the solid path of  $G_n - g$ in Fig.~16 (or Fig.~17, respectively). Hence, we  have a recurrence construction modulo $5$ of the snake path  $\mathbb{W}_n$ (or $\mathbb{Z}_n$) of $(G_n, g)$, for $n \geqslant 5$, satisfying the condition of Lemma \ref{lem4.3}.
\qed \end{proof}


\begin{definition}\label{def4.2}
For $1 \leqslant i \leqslant 2$, $1 \leqslant k \leqslant 4$ and some $0 \leqslant j \leqslant 5$, we define a pattern
$[\Psi^{(i, j)}_k, d, C^{(i, j)}_k]$, where $\Psi^{(i, j)}_k$ is a plane semi-triangulation in Fig.~18 (for $i = 2$), or Fig.~19 (for $i = 1$), $d$ is a fixed vertex in $\sigma(\Psi^{(i, j)}_k)$, and $C^{(i, j)}_k$ is a fixed set of vertices marked by circles in the graph $\Psi^{(i, j)}_k$.

\end{definition}
 The following lemma follows directly from Definition \ref{def4.2}.
\begin{lemma}\label{lem4.4}
$\phantom{xxx}$
\begin{enumerate}[\upshape(1)]
\item
$[\Psi^{(i, j)}_k, d, C^{(i, j)}_k]$ is $\sigma$-compatible with $[\Psi^{(2, 0)}_k, d, C^{(2, 0)}_k]$.
\item
Every  set $C^{(i,j)}_k$ dominates all faces of the graph $\Psi^{(i, j)}_k$ and $d \notin C^{(i, j)}_k$,
\item
 Every induced graph $\Psi^{(2, j)}_k [C^{(2, j)}_k]$ (or $\Psi^{(1, j)}_k [C^{(1, j)}_k]$) is a tree (consists of two paths, respectively).
\end{enumerate}
\end{lemma}

 \begin{lemma}\label{lem4.5}
 Every graph $(G, g) \in  {\cal{G}}_0$, $(G, g) \neq (H_7, g)$, which is generated by operations of type $C$ and $B$ from one of graphs $(A(G_n), g)$, $n \geqslant 6$,  has a hamiltonian set which is $(\pm$)compatible with the orientation of $c(G)$.

 The graph $(H_7, g)$,  has a hamiltonian set of vertices, marked by circles in Fig.~20, which is compatible with the orientation of $c(H_7)$.
  \end{lemma}

  \begin{proof}
Fix $n \geqslant 6$. Let $u_1 v_1$, $u_2 v_2$ be two distinguished edges of the graph $G_n - g$. Recall that $u_1v_1$ (or $u_2 v_2$) is the first edge on the left side (right side, respectively) of $G_n - g$. Suppose that $\mathbb{W}_n$, $\mathbb{Z}_n$ are two snake paths of $(G_n, g)$ such that $\alpha_i(\mathbb{W}_n)$ and $\alpha_i(\mathbb{Z}_n)$,  for $i = 1, 2$, satisfy conditions of Lemma \ref{lem4.3}.

Since $G_n -g$ is a configuration of $(A(G_n), g)$ and $c^2(A(G_n)) = c(G_n)$, we can assume that $u_1 v_1$, $u_2 v_2$ are edges of the cycle $c^2(A(G_n))$ such that $u_i$ and $v_i$ has, respectively, degree $4$  and $5$ in $A(G_n)$. Similarly, we can assume that $\mathbb{W}_n$, $\mathbb{Z}_n$ are induced paths of $A(G_n)$.

Let us denote
\begin{align*}
W_n &= \{v \in V(A(G_n))\colon d(v, \mathbb{W}_n) = 1\},
\cr
Z_n &= \{ v \in V(A(G_n))\colon d(v, \mathbb{Z}_n) = 1\}.
\end{align*}
It is easy to see that:
 \begin{enumerate}
\item [(1)]
both $W_n$  and $Z_n$  dominate all faces of $A(G_n)$,
\item [(2)]
both $A(G_n)[W_n]$ and $A(G_n)[Z_n]$ are cycles in $A(G_n)$.
\item [(3)]
both $A(G_n)[W_n]\cap c(A(G_n))$ and $A(G_n)[Z_n]\cap c(A(G_n))$ do not contain an isolated vertex.
\end{enumerate}

Suppose that $t_i$, for $i = 1, 2$, is a vertex of the cycle $c(A(G_n))$ which is adjacent to the vertices $u_{i}$ and $v_i$ of the cycle $c^2(A(G_n))$. Assume that $\Phi_i$ is a minimal configuration of $A(G_n)$ such that $u_i$, $v_i$, $t_i$ are the only vertices in the proper region of $\Phi_i$.

By the definition of $\alpha_i(\mathbb{W}_n)$, $\alpha_i(\mathbb{Z}_n)$, for $i = 1, 2$, and by the definition of patterns $[\Psi^{(2, 0)}_k, d, C^{(2, 0)}_k]$, for $k = 1, \ldots, 4$, it follows that
\begin{enumerate}
\item[(4)]
 $\left\{
 \begin{array}{ll}
[\Phi_1, g, W_n \cap V(\Phi_1)] = [\Psi^{(2, 0)}_{\alpha_1(\mathbb{W}_n)}, d, C^{(2, 0)}_{\alpha_1(\mathbb{W}_n)}],
\\[4pt]
[\Phi_2, g, W_n \cap V(\Phi_2)] = [\Psi^{(2, 0)}_{\alpha_2(\mathbb{W}_n)}, d, C^{(2, 0)}_{\alpha_2(\mathbb{W}_n)}], \hbox{ for } n  \hbox{ even},
\\[4pt]
[\Phi_2, g, W_n \cap V(\Phi_2)] = [r(\Psi^{(2, 0)}_{\alpha_2(\mathbb{W}_n)}), d, r(C^{(2, 0)}_{\alpha_2(\mathbb{W}_n)})], \hbox{ for } n  \hbox{ odd}, \end{array}
 \right.$
\end{enumerate}
and
\begin{enumerate}
\item[(5)]
 $\left\{
 \begin{array}{ll}
[\Phi_1, g, Z_n \cap V(\Phi_1)] = [\Psi^{(2, 0)}_{\alpha_1(\mathbb{Z}_n)}, d, C^{(2, 0)}_{\alpha_1(\mathbb{Z}_n)}],
\\[4pt]
[\Phi_2, g, Z_n \cap V(\Phi_2)] = [\Psi^{(2, 0)}_{\alpha_2(\mathbb{Z}_n)}, d, C^{(2, 0)}_{\alpha_2(\mathbb{Z}_n)}], \hbox{ for } n  \hbox{ even},
\\[4pt]
[\Phi_2, g, Z_n \cap V(\Phi_2)] = [r(\Psi^{(2, 0)}_{\alpha_2(\mathbb{Z}_n)}), d, r(C^{(2, 0)}_{\alpha_2(\mathbb{Z}_n)})], \hbox{ for } n \hbox{ odd}, \end{array}
 \right.$
\end{enumerate}
where $r(\Psi^{(2, 0)}_{\alpha_2(\mathbb{W}_n)})$ (or $r(\Psi^{(2, 0)}_{\alpha_2(\mathbb{Z}_n)})$) is a mirror reflection of $\Psi^{(2, 0)}_{\alpha_2(\mathbb{W}_n)}$ (or $\Psi^{(2, 0)}_{\alpha_2(\mathbb{Z}_n)}$, respectively).

Since $n \geqslant 6$, we obtain:
\begin{enumerate}
\item[(6)]
proper regions of $\Phi_1$ and $\Phi_2$ are disjoint.
\end{enumerate}

Let $\Phi_{W}(n)$ be the set of all functions $\bar{w} :  \{1,2\} \to \{ 0,\ldots, 5\}$ satisfying the following conditions:
\begin{align*}
&\bar{w}(1) \in \{ 0, 5\} \hbox{ and } \bar{w}(2)\in \{0, 1\},&\hbox{for } n\equiv 1 \pmod5,
\cr
&\bar{w}(1) = \bar{w}(2) = 0, &\hbox{for } n \equiv 2 \pmod5,
\cr
&\bar{w}(1) \in \{ 1, 2, 4\} \hbox{ and } \bar{w}(2) \in \{0, 1, 4\},&\hbox{for } n\equiv 3\pmod5,
\cr
&\bar{w}(1) \in \{1, 4\} \hbox{ and } \bar{w}(2) \in \{1, 4\},&\hbox{for } n\equiv 4\pmod5,
\cr
&\bar{w}(1) = \bar{w}(2)= 4, &\hbox{for } n\equiv 0 \pmod5.
\end{align*}
Similarly, let $\Phi_{Z}(n)$ be the set of all functions $\bar{z}:  \{1,2\} \to \{ 0,\ldots, 5\}$ satisfying the following conditions:
\begin{align*}
&\bar{z}(1) \in \{1, 2, 4\} \hbox{ and }  \bar{z}(2) = 0, &\hbox{for } &n\equiv 1, \hbox{ or } n\equiv 2\pmod5,
\cr
&\bar{z}(1) \in \{0, 3, 5\} \hbox{ and }  \bar{z}(2) \in \{0, 1, 5\}, &\hbox{for } &n\equiv 3 \pmod5,
\cr
&\bar{z}(1)= \bar{z}(2) = 0, &\hbox{for } &n\equiv 4, \hbox{ or } n\equiv 0\pmod5.
\end{align*}

Fix $\bar{w} \in \Phi_{W}(n)$ and  $\bar{z} \in \Phi_{Z}(n)$. By (4)--(6) and Lemma \ref{lem4.4}(1), we can define replacement operations
\begin{align*}
w:  &[\Phi_1, g, W_n \cap V(\Phi_1)] \xrightarrow
\
[\Psi^{(1, \bar{w}(1))}_{\alpha_1(\mathbb{W}_n)}, d, C^{(1, \bar{w}(1))}_{\alpha_1(\mathbb{W}_n)}],
\cr
&[\Phi_2, g, W_n \cap V(\Phi_2)] \xrightarrow
\
[\Psi^{(2, \bar{w}(2))}_{\alpha_2(\mathbb{W}_n)}, d, C^{(2, \bar{w}(2))}_{\alpha_2(\mathbb{W}_n)}], \hbox{ for } n  \hbox{ even},
\cr
&[\Phi_2, g, W_n \cap V(\Phi_2)] \xrightarrow
\
[r(\Psi^{(2, \bar{w}(2))}_{\alpha_2(\mathbb{W}_n)}), d, r(C^{(2, \bar{w}(2))}_{\alpha_2(\mathbb{W}_n)})], \hbox{ for } n \hbox{ odd},
\end{align*}
and
\begin{align*}
z:  &[\Phi_1, g, Z_n \cap V(\Phi_1)] \xrightarrow
\
[\Psi^{(1, \bar{z}(1))}_{\alpha_1(\mathbb{Z}_n)}, d, C^{(1, \bar{z}(1))}_{\alpha_1(\mathbb{Z}_n)}],
\cr
&[\Phi_2, g, Z_n \cap V(\Phi_2)] \xrightarrow
\
[\Psi^{(2, \bar{z}(2))}_{\alpha_2(\mathbb{Z}_n)}, d, C^{(2, \bar{z}(2))}_{\alpha_2(\mathbb{Z}_n)}], \hbox{ for } n  \hbox{ even},
\cr
&[\Phi_2, g, Z_n \cap V(\Phi_2)] \xrightarrow
\
[r(\Psi^{(2, \bar{z}(2))}_{\alpha_2(\mathbb{Z}_n)}), d, r(C^{(2, \bar{z}(2))}_{\alpha_2(\mathbb{Z}_n)})], \hbox{ for } n  \hbox{ odd},
\end{align*}
to obtain plane triangulations $(w(A(G_n)),g)$ and $(z(A(G_n)),g) \in \mathcal{G}$,  and the following sets $w(W_n)$,  $z(Z_n) \subset V(w(A(G_n)))$
\begin{align*}
w(W_n) &= (W_n \setminus \bigcup_{i = 1, 2} V(\Phi_i)) \cup  \bigcup_{i = 1, 2}C^{(i, \bar{w}(i))}_{\alpha_i(\mathbb{W}_n)}, \hbox{ for } n  \hbox{ even},
\cr
w(W_n) &= (W_n \setminus \bigcup_{i = 1, 2} V(\Phi_i)) \cup  C^{(1, \bar{w}(1))}_{\alpha_1(\mathbb{W}_n)} \cup  r(C^{(2, \bar{w}(2))}_{\alpha_2(\mathbb{W}_n)}), \hbox{ for } n  \hbox{ odd},
 \end{align*}
 and
 \begin{align*}
 z(Z_n) &= (Z_n \setminus \bigcup_{i = 1, 2} V(\Phi_i)) \cup  \bigcup_{i = 1, 2}C^{(i, \bar{z}(i))}_{\alpha_i(\mathbb{Z}_n)},\hbox{ for } n  \hbox{ even},
 \cr
z(Z_n) &= (Z_n \setminus \bigcup_{i = 1, 2} V(\Phi_i)) \cup  C^{(1, \bar{z}(1))}_{\alpha_1(\mathbb{Z}_n)} \cup  r(C^{(2, \bar{z}(2))}_{\alpha_2(\mathbb{Z}_n)}), \hbox{ for } n  \hbox{ odd}.
\end{align*}

Recall that $\mathbb{W}_n$ and $\mathbb{Z}_n$ have recurrence constructions modulo $5$, and $\alpha_1(\mathbb{W}_n)$, $\alpha_2(\mathbb{W}_n)$,  $\alpha_1(\mathbb{Z}_n)$, $\alpha_2(\mathbb{Z}_n)$ are periodic functions with period $5$. Hence, it follows that
 \begin{enumerate}
\item [(7)]
$(w(A(G_n)),g)$ and $(z(A(G_n)), g)$ (similarly, the sets $w(W_n)$ and $z(Z_n)$) have recurrence constructions modulo $5$.
\end{enumerate}

By (3), (4) and Lemma~\ref{lem4.4}(1), $w(W_n)$  is $(\pm)$compatible with the orientation of $c(w(A(G_n))$.  By (1), (4) and Lemma~\ref{lem4.4}(2), $w(W_n)$ dominates all faces of $w(A(G_n))$. By (2), (4) and Lemma \ref{lem4.4}(3), $w(W_n)$ induces a tree in $w(A(G_n))$. 
Similarly, $z(Z_n)$ is $(\pm)$compatible with the orientation of $c(z(A(G_n))$, dominates all faces, and induces a tree in $z(A(G_n))$. Hence, we obtain
\begin{enumerate}
\item [(8)]
$w(W_n)$ (or $z(Z_n)$) is a hamiltonian set which is $(\pm)$compatible with the orientation of $c(w(A(G_n)))$ (or $c(z(A(G_n)))$, respectively).
\end{enumerate}

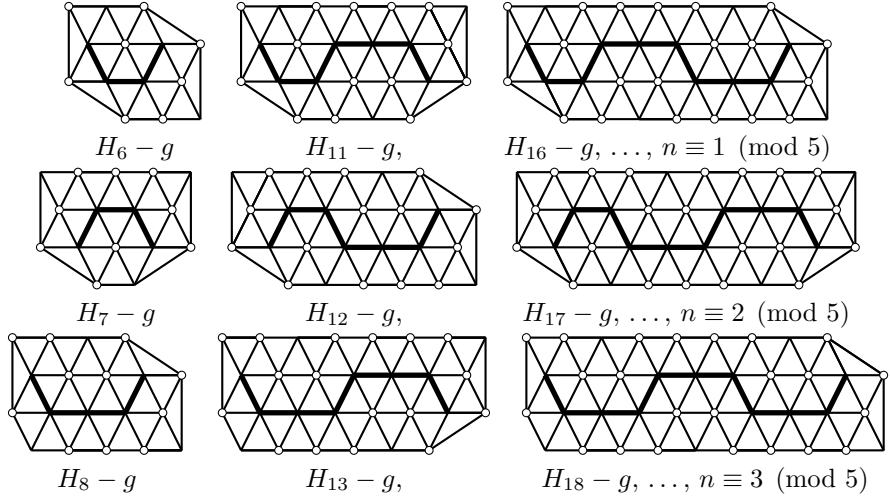
Let  $(H_n, g) \in {\cal{G}}_0$, for $n = k + 5j$, where $k = 6, 7, 8$ and $j = 0, 1, 2$, be a graph defined by the graph $H_n - g \in \cal{F}$ in Fig.~20.  Notice that $(H_n, g)$, for $n \neq 7$, has a hamiltonian set of vertices, marked by circles in Fig.~20, which is $(\pm)$compatible with the orientation of $c(H_n)$. Hence, we obtain a recurrence construction of $(H_n, g) \in {\cal{G}}_0$, for $n \geqslant 6$, $n\equiv1 \pmod5$ ($n\equiv2,3 \pmod5$, respectively), as well as a recurrence construction of a hamiltonian set in $(H_n, g)$ which is $(\pm)$compatible (for $n \neq 7$), or compatible (for $n = 7$) with the orientation of $c(H_n)$. 




Put $(H, g) = (w(A(G_n)), g)$ and $X = w(W_n)$ (or $(H, g) = (z(A(G_n)),g)$ and $X = z(Z_n)$). As in Definition \ref{def3.2} we define the set $\Omega^*(H, X)$,  the replacement operation $\omega$, and the plane triangulation $(\omega(H), g)$   for $\bar{\omega} \in \Omega^*(H, X)$.

Let  ${\cal{S}}_n \subset {\cal{G}}_0$ be the family of all graphs which are generated by operations  of type $C$ and $B$ from $(A(G_n), g)$. By the definition of the replacement operation $w$ and $z$, $(w(A(G_n)),g)$ and $(z(A(G_n)), g) \in {\cal{S}}_n$.

Observe that, for an even $n$, there exists exactly $\binom{8}{2} + 8 = 36$ different  graphs in ${\cal{S}}_n$. Similarly, for an odd $n$, there exists exactly $2\binom{8}{2} +8 = 64$ different graphs in ${\cal{S}}_n$. Hence, by the definition of $\omega$, it is not difficult to check that every graph $(G, g) \in {\cal{S}}_n$, where $6 \leqslant n \leqslant 10$, is op-equivalent to one of the following graphs (or to a mirror reflection of one of them, for $n$ odd):
\begin{align*}
&(\omega(w(A(G_n))), g), \hbox{ for } \bar{w} \in \Phi_{W}(n) \hbox{ and }\bar{\omega} \in \Omega^*(w(A(G_n)), w(W_n)),
 \cr
&(\omega(z(A(G_n))), g),  \hbox{ for } \bar{z} \in \Phi_{Z}(n) \hbox{ and } \bar{\omega} \in \Omega^*(z(A(G_n)), z(Z_n)),
\cr
\hbox{or}
\cr
&(H_n, g), \hbox{ for } n \equiv 1, 2,  3 \pmod5.
\end{align*}
From (7), it follows that every graph $(G, g) \in {\cal{S}}_n$, for $n > 10$, is also op-equivalent to one of the above graphs (or to a mirror reflection of one of them, for $n$ odd).
Finally, by (8) and Lemma \ref{lem3.2}, every graph $(G, g) \in {\cal{S}}_n$,  different than $ (H_7, g)$, has a hamiltonian set  which is ($\pm$)com\-pat\-i\-ble with the orientation of $c(G)$.
 \qed \end{proof}
\begin{lemma}\label{lem4.6}
Every graph $(G, g) \in  {\cal{G}}_n$, for $n \geqslant 1$, of height $2$ has a hamiltonian set of vertices which is $(\pm)$compatible with the orientation of $c(G)$.
\end{lemma}

\begin{proof}
We know that every graph $(F_n, g)$, $n \geqslant 1$, has a hamiltonian set of vertices, say  $Y_n$, marked by circles in Fig.~5, which is $(\pm)$compatible with the orientation of $c(F_n)$

We define a graph $(B^j_1, g)\in  {\cal{G}}_1$,  where $j = 1, 2, 3$, as the graph $B^j_1 - g \in  \cal{F}$ shown in Fig.~21. $(B^j_1, g)$ has a hamiltonian set of vertices, say $X^j_1$, marked by circles in Fig.~21, which is $(\pm)$compatible with the orientation of $c(B^j_1)$.

We also define a graph $(B_n, g) \in  {\cal{G}}_n$, for $n \geqslant 2$, as the graph $B_n - g \in  \cal{F}$ shown in Fig.~21. $(B_n, g)$ has a hamiltonian set of vertices , say $X_n$, marked by circles in Fig.~21, which is $(\pm)$compatible with the orientation of $c(B_n)$. 

Let us consider the graph $(B_n, g)$, for $n \geqslant 2$, ($(F_1, g)$, or $(B_1^1, g)$) with the hamiltonian set $X_n$ ($Y_1$, or $X_1^1$, respectively). As in Definition~\ref{def3.2} we define sets $M_1, \ldots , M_4$.  Notice that $M_1 = M_4 = \emptyset$ and $M_2 \cap M_3 = \emptyset$. As in Definition~\ref{def3.2} we define  $\Omega^*(B_n, X_n)$ ($\Omega^*(F_1, Y_1)$, or $\Omega^*(B_1^1, X_1^1)$), the replacement operation $\omega$, and the plane triangulation $(\omega(B_n), g)$ ($(\omega(F_1), g)$, or $(\omega(B_1^1), g)$, respectively).

Let $(G, g) \in  {\cal{G}}_n$, for $n \geqslant 1$, be a graph of height $2$. By Theorem \ref{theo2.1} it is generated by operations of type $C$ and $B$ from $(F_n, g)$. Hence, by the definition of $\omega$, it is easy to check that if $n = 1$, then the graph $(G, g)$ is op-equivalent to one of the following graphs:
\begin{align*}
&(\omega(F_1), g),  \hbox{ for } \bar{\omega} \in \Omega^*(F_1, Y_1),
\cr
&(\omega(B_1^1), g), \hbox{ for } \bar{\omega} \in \Omega^*(B_1^1, X_1^1),
\cr
&(B^2_1, g),  \hbox{ or } (B^3_1, g).
\end{align*}
By analogy, we check that if $n  \geqslant 2$, then the graph $(G, g)$ is op-equivalent to one of the following graphs:
\begin{align*}
&(F_n, g), 
\cr
\hbox{or}
\cr
&(\omega(B_n), g), \hbox{ for }  \bar{\omega} \in \Omega^*(B_n, X_n).
\end{align*}
Thus, by Lemma \ref{lem3.2}, $(G, g)$ has  a hamiltonian set of vertices which is $(\pm)$com\-pat\-i\-ble with the orientation of the cycle $c(G)$.
\qed \end{proof}

 \begin{definition}\label{def4.3}
For $k = 5$, $6$ and $j = 0$, $1$, we define a pattern
$[\Delta^j_k, d, D^j_k]$, where $\Delta^j_k$ is a plane semi-triangulation in Fig.~22, $d$ is a fixed vertex in $\sigma(\Delta^j_k)$, and $D^j_k$ is a fixed set of vertices marked by circles in the graph $\Delta^j_k$.


\end{definition}

The following lemma follows directly from Definition \ref{def4.3}.

\begin{lemma}\label{lem4.7}
For $k = 5, 6$ and $j = 0, 1$ the following conditions are satisfied:
\begin{enumerate}[\upshape(1)]
\item
$[\Delta^1_k, d, D^1_k]$ is $\sigma$-compatible with $[\Delta^0_k, d, D^0_k]$,
\item
$D^j_k$ dominates all faces of $\Delta^j_k$, and $d \notin D^j_k$.
\end{enumerate}
\end{lemma}
\begin{lemma}\label{lem4.8}
Every graph $(G, g) \in {\mathcal{G}}_n$, for $n \geqslant 1$, of height $3$ and different than $(Q, g)$,  has a hamiltonian set of vertices which is compatible with the orientation of $c(G)$.
\end{lemma}

\begin{proof}
Notice that every graph $(A(F_n), g)$, for $n \geqslant 3$, has a hamiltonian set of vertices, marked by circles in Fig.~23 (for $n = 3$), which is $(\pm)$compatible with the orientation of $c(A(F_n))$.

We define a graph $(F^j_1, g) \in  {\cal{G}}_1$, where $1 \leqslant  j \leqslant 6$, as the graph $F^j_1 - g \in  \cal{F}$ shown in Fig.~24. We define also a graph $(F^j_n, g) \in  {\cal{G}}_n$, for $n \geqslant 2$ and $1 \leqslant  j \leqslant 11$, as the graph $F^j_n - g \in  \cal{F}$ shown in Fig.~25--28 (for $n = 2, 3$). Remark that $(F^1_1, g) = (A(F_1), g)$ and $(F^1_2, g) = (A(F_2), g)$.




Notice that every graph $(F^j_n, g)$ has a hamiltonian set of vertices, say $W^j_n$, marked by circles in Fig.~$24-28$, which is ($\pm$)compatible with the orientation of $c(F^j_n)$.

Let $\Theta_0 \in \cal{F}$ be a configuration of $F^1_1$, or $F^j_n$, for $n \geqslant 2$ and $1 \leqslant j \leqslant 6$, which has just four proper faces $gx_1t_1$, $x_1t_1y_1$, $t_1y_1t_2$, $y_1t_2t_3$ (see Fig.~$24-26$). Let $\Theta_i \in \cal{F}$ be a configuration of $F^j_1$, where $1\leqslant i \leqslant j \leqslant 3$, which has just three proper faces $gx_iv_i$, $x_iv_iz_i$, $x_iz_iy_i$ (see Fig.~24). Furthermore, let $\Theta_2 \in \cal{F}$ be a configuration of $F^j_n$, for $n \geqslant 2$ and $2 \leqslant j \leqslant 5$, which has just three proper faces $gx_2v_2$, $x_2v_2z_2$, $x_2z_2y_2$ (see Fig.~$25-26$). By Definition \ref{def3.3} of $[\Delta^0_2, d, D^0_2]$, and, by  Definition \ref{def4.3} of $[\Delta^0_5, d, D^0_5]$ and  $[\Delta^0_6, d, D^0_6]$, we obtain
\begin{enumerate}
\item[(1)]
 $\left\{
 \begin{array}{ll}
 [\Theta_0, g, W^j_n \cap V(\Theta_0)] = [\Delta^0_2, d, D^0_2],  &\hbox{for } \ n = 1, j = 1, \hbox{ and }
 \\
 &\hbox{for } \ n \geqslant 2,\ 1 \leqslant j \leqslant 6,
\\[4pt]
 [\Theta_i, g, W^j_1 \cap V(\Theta_i)] = [\Delta^0_5, d, D^0_5], &\hbox{for } \ 1\leqslant i \leqslant j \leqslant 3,
    \\[4pt]
 [\Theta_2, g, W^j_n \cap V(\Theta_2)] = [\Delta^0_5, d, D^0_5], &\hbox{for } \ n \geqslant 2,\ j = 2, 4,
    \\[4pt]
[\Theta_2, g, W^j_n \cap V(\Theta_2)] = [\Delta^0_6, d, D^0_6], &\hbox{for } \ n \geqslant 2,\ j  = 3, 5.
\end{array}
 \right.$
\end{enumerate}
Certainly, by the definition of $\Theta_i$, $i = 0, 1, 2, 3$, we have
\begin{enumerate}
\item[(2)]
 $\left\{
 \begin{array}{l}
\hbox{proper regions of } \Theta_0 \hbox{ and } \Theta_1  \hbox{ are disjoint in the graph } F^1_1,
 \\[4pt]
 \hbox{proper regions of } \Theta_1 \hbox{ and } \Theta_2  \hbox{ are disjoint in the graph } F^2_1,
 \\[4pt]
   \hbox{proper regions of } \Theta_1, \Theta_2 \hbox{ and } \Theta_3 \hbox{ are paiwise disjoint}
\\
 \hbox{in the graph }  F^3_1,
 \\[4pt]
\hbox{proper regions of } \Theta_0 \hbox{ and } \Theta_2 \hbox{ are disjoint in the graph } F^j_n,
\\
\hbox{for } n \geqslant 2 \hbox{ and }  2 \leqslant j \leqslant 5.
  \end{array}
 \right.$
 \end{enumerate}

Let $\Theta$ be the set of all functions $\bar{\theta} \colon \{0,1,2,3 \} \to \{ 0, 1, \ldots , 6 \}$ such that $\bar{\theta}(i) \leqslant 1$, for $i \neq 0$.

Fix $\bar{\theta} \in \Theta$. Let us consider $(F^j_1, g)$, where $1 \leqslant j \leqslant 3$ (or $(F^j_n, g)$,  for $n \geqslant 2$ and $1 \leqslant j \leqslant 6$). By (1), (2) and Lemma~\ref{lem4.7}(1), we can define the following replacement operation $\theta = \theta^j_n$ to obtain a plane triangulation $(\theta(F^j_n), g)$ from $(F^j_n, g)$, and the following set $\theta(W^j_n) \subset V(\theta(F^j_n))$ from $W^j_n$:
\\[4pt]
For $n = 1$ and $j = 1$,
\begin{align*}
\theta:  &[\Theta_0, g, W^1_1 \cap V(\Theta_0)] \xrightarrow
 \
[\Delta^{\bar\theta(0)}_2, d, D^{\bar\theta(0)}_2 ],
\cr
&[\Theta_1, g, W^1_1 \cap V(\Theta_1)] \xrightarrow
 \
[\Delta^{\bar\theta(1)}_5, d, D^{\bar\theta(1)}_5 ],
\end{align*}
$$\theta(W^1_1) = (W^1_1 \setminus (V(\Theta_0)\cup V(\Theta_1)) \cup D^{\bar\theta(0)}_2 \cup D^{\bar\theta(1)}_5 \subset V(\theta(F^1_1)).$$
For $n = 1$ and $j = 2$,
\begin{align*}
\theta:  &[\Theta_i, g, W^2_1 \cap V(\Theta_k)] \xrightarrow {i = 1, 2}
 \
[\Delta^{\bar\theta(i)}_5, d, D^{\bar\theta(i)}_5 ],
\end{align*}
$$\theta(W^2_1) = (W^2_1 \setminus (V(\Theta_1)\cup V(\Theta_2)) \cup D^{\bar\theta(1)}_5 \cup D^{\bar\theta(2)}_5 \subset V(\theta(F^2_1)).$$
For $n = 1$ and $j = 3$,
\begin{align*}
\theta:  &[\Theta_i, g, W^3_1 \cap V(\Theta_i)] \xrightarrow {i = 1, 2, 3}
 \
[\Delta^{\bar\theta(i)}_5, d, D^{\bar\theta(i)}_5 ],
\end{align*}
$$\theta(W^3_1) = (W^3_1 \setminus \bigcup_{i = 1, 2, 3} V(\Theta_i)) \cup \bigcup_{i = 1, 2, 3}D^{\bar\theta(i)}_5 \subset V(\theta(F^3_1)).$$
For $n \geqslant 2$ and $j = 1, 6$,
\begin{align*}
\theta:  &[\Theta_0, g, W^j_n \cap V(\Theta_0)] \xrightarrow
 \
[\Delta^{\bar\theta(0)}_2, d, D^{\bar\theta(0)}_2 ],
\end{align*}
$$\theta(W^j_n) = (W^j_n \setminus V(\Theta_0)) \cup D^{\bar\theta(0)}_2 \subset V(\theta(F^j_n)).$$
For $n \geqslant 2$ and $j = 2, 4$,
\begin{align*}
\theta:  &[\Theta_0, g, W^j_n \cap V(\Theta_0)] \xrightarrow
 \
[\Delta^{\bar\theta(0)}_2, d, D^{\bar\theta(0)}_2 ],
\cr
&[\Theta_2, g, W^j_n \cap V(\Theta_2)] \xrightarrow
 \
[\Delta^{\bar\theta(2)}_5, d, D^{\bar\theta(2)}_5 ],
\end{align*}
$$\theta(W^j_n) = (W^j_n \setminus (V(\Theta_0)\cup V(\Theta_2)) \cup (D^{\bar\theta(0)}_2 \cup D^{\bar\theta(2)}_5) \subset V(\theta(F^j_n)).$$
For $n \geqslant 2$ and $j = 3, 5$,
\begin{align*}
\theta:  &[\Theta_0, g, W^j_n \cap V(\Theta_0)] \xrightarrow
 \
[\Delta^{\bar\theta(0)}_2, d, D^{\bar\theta(0)}_2 ],
\cr
&[\Theta_2, g, W^j_n \cap V(\Theta_2)] \xrightarrow
 \
[\Delta^{\bar\theta(2)}_6, d, D^{\bar\theta(2)}_6 ],
\end{align*}
$$\theta(W^j_n) = (W^j_n \setminus (V(\Theta_0)\cup V(\Theta_2)) \cup (D^{\bar\theta(0)}_2 \cup D^{\bar\theta(2)}_6) \subset V(\theta(F^j_n)).$$
By the definition of $\theta = \theta^j_n$,  $(\theta(F^j_n),g) \in {\cal{G}}_n$. 
Since $W^j_n$  is a hamiltonian set in $(F^j_n, g)$ which is $(\pm)$compatible with the orientation of $c(F^j_n)$, by Lemma~\ref{lem3.3} and Lemma \ref{lem4.7}, we obtain the following:
\begin{enumerate}
\item[(3)]
$\theta(W^j_n)$  is a hamiltonian set in $(\theta(F^j_n), g)$ which is $(\pm)$compatible with the orientation of $c(\theta(F^j_n))$.
\end{enumerate}

Put $(H, g) = \Omega(\theta(F^j_n))$ and $X = \theta(W^{j}_n))$ in Definition \ref{def3.2}. As in Definition~\ref{def3.2}, we define the set of functions $\Omega(H, X)$,  the replacement operation $\omega$, and the plane triangulation $(\omega(H, g)$, for $\bar{\omega} \in \Omega(H, X)$.

Let $(G, g) \in {\mathcal{G}}_n$, for $n \geqslant 1$, be a graph of height $3$. By Theorem \ref{theo2.2}, it is generated by operations of type $\bar{C}$, $\bar{B}$, $C$, $B$ from $(A(F_n), g)$.
Hence, by the definitions of $\theta = \theta^j_1$ and $\omega$, it is easy to check that if $n = 1$, then the graph $(G, g)$ is op-equivalent to one of the following graphs:
\begin{align*}
&(\omega (\theta(F^j_1)), g),  \hbox{ for } \bar{\theta} \in \Theta,\ \bar{\omega} \in \Omega(\theta(F^j_1), \theta(W^j_1)), \hbox{ and }1 \leqslant j \leqslant 3,
\cr
&(F^j_1, g), \hbox{ where } 4 \leqslant j \leqslant 6,
\cr
\hbox{or}
\cr
&(Q, g).
\end{align*}
 Similarly, by the definitions of $\theta = \theta^j_n$ and $\omega$, it is easy to check that if $n \geqslant 2$, then the graph $(G, g)$ is op-equivalent to one of the following graphs:
\begin{align*}
&(A(F_n), g), \hbox{ for } n \geqslant 3, 
\cr
&(\omega (\theta(F^j_n)), g), \hbox{ for } \bar{\theta} \in \Theta,\ \bar{\omega} \in \Omega(\theta(F^j_n), \theta(W^j_n)), \hbox{ and }1 \leqslant j \leqslant 6,
\cr
\hbox{or}
\cr
&(F^j_n, g), \hbox{ where } 7 \leqslant j \leqslant 11.
\end{align*}
Thus, by (3) and Lemma \ref{lem3.2}, the graph $(G, g)$, different than $(Q, g)$, has a hamiltonian set of vertices which is ($-$)compatible with the orientation of $c(G)$. Finally, by Lemma~\ref{lem3.9}, it has a hamiltonian set which is compatible with the orientation of $c(G)$, because $(r(F_n), g) = (F_n, g)$, for $n \geqslant 1$.
\qed \end{proof}

\begin{theorem}\label{theo4.1}
Every graph $(G, g) \in {\cal {G}}_n$ of height $h(G) \leqslant 3$, for $n \geqslant 1$ (or $h(G) \leqslant 2$, for $n = 0$), and different than $(P, g)$ and $(Q, g)$, has a hamiltonian set of vertices which is compatible with the orientation of $c(G)$.
\end{theorem}

\begin{proof}
Theorem \ref{theo4.1} follows from Corollary \ref{coro2.1}(2) and from Lemmas \ref{lem4.1}, \ref{lemma4.2}, \ref{lem4.5}, \ref{lem4.6} and \ref{lem4.8}.
\qed \end{proof}

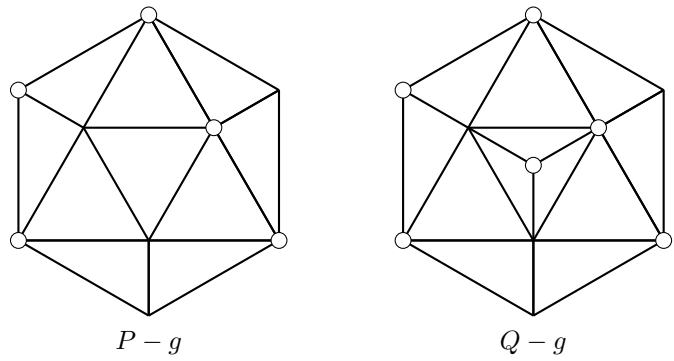
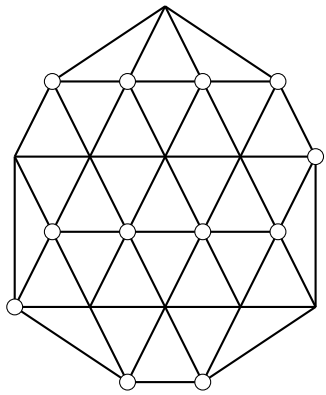
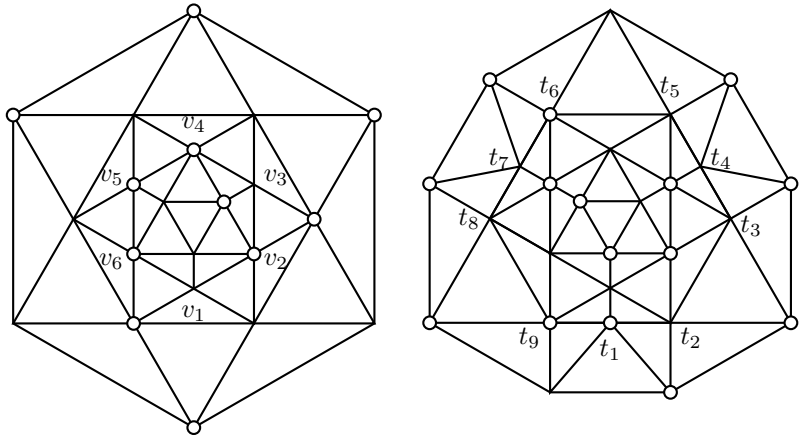
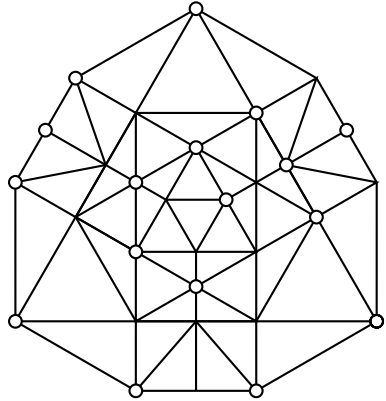
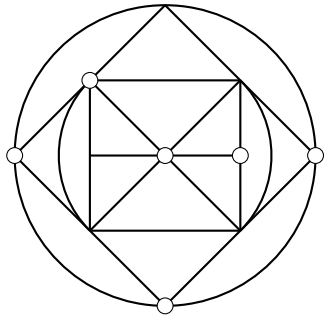
\end{document}